\documentclass[10pt]{elsarticle}

\usepackage{amsmath, amsfonts, amsthm, amssymb}
\usepackage[scr=rsfso]{mathalfa}
\usepackage{xparse}

\numberwithin{equation}{section}

\theoremstyle{remark}
\newtheorem*{remk}{Remark}

\theoremstyle{plain} 
\newtheorem{thrm}{Theorem}[section]
\newtheorem{lem} [thrm]{Lemma}
\newtheorem{prop}[thrm]{Proposition}
\newtheorem{cor} [thrm]{Corollary}

\theoremstyle{definition}
\newtheorem{defn}{Definition}[section]

\usepackage[a4paper]{geometry}
\NewDocumentCommand{\norm}{m m} {\IfNoValueTF{#1} {\left\Vert{#2}\right\Vert} {\left\Vert{#2}\right\Vert_{#1}}}

\begin{document}

\begin{abstract}
We consider the defocusing energy-critical inhomogeneous nonlinear Schr\"{o}dinger equation (INLS) $iu_t + \Delta u = |x|^{-b}|u|^{k}u$ in $\mathbb{R} \times \mathbb{R}^{n}$ where $n \geq 3$, $0<b<\min(2, n/2)$, and $k=(4-2b)/(n-2)$.
We show that for every spherically symmetric initial data $\phi$ in $H^1(\mathbb{R}^n)$, or preferably in $\dot{H}^1(\mathbb{R}^n)$, the associated solution is globally well-posed and scatters for every such $n$ and $b$ except for $n=4$ with $1\leq b<2$ and $n=5$ with $1/2\leq b\leq 5/4$.
We mainly apply the arguments of Tao (2005), but inspired by the work of Aloui and Tayachi (2021), we utilize Lorentz spaces to define spacetime norms.
This method is distinct from the concentration compactness principle and establishes a quantitative bound for the solution's spacetime norm.
The bound has an exponential form $C\exp(CE[\phi]^C)$ in terms of the energy $E[\phi]$, similar to Tao's work.
\end{abstract}


\begin{frontmatter}
\title{Global well-posedness and scattering of the defocusing energy-critical inhomogeneous nonlinear Schr\"{o}dinger equation with radial data}

\author[1]{Dongjin Park}
\ead{aseastar@kaist.ac.kr}

\affiliation[1]{
	organization={Department of Mathematical Sciences, Korea Advanced Institute of Science and Technology},
	city={Daejeon},
	postcode={34141},
	country={Republic of Korea},
}
\end{frontmatter}


\section{Introduction} \label{Section:1}

Let $n\geq 3$ be the dimension, $0<b<\min(2, n/2)$, $\mu = \pm 1$, and $k>0$. We are interested in the solution $u:I\times \mathbb{R}^n \to \mathbb{C}$ of the inhomogeneous nonlinear Schr\"{o}dinger equations (shortened as INLS) defined as below. When $k > 1$, we would want to assume a weaker initial condition $\phi \in \dot{H}^1(\mathbb{R}^n)$, which is explained later in Theorems \ref{thrm:main} and \ref{cor:scattering}, and Subsection \ref{Subsection:2.2}.
\begin{equation} \label{inls}
\left\{
\begin{aligned}
&iu_t + \Delta u = \mu |x|^{-b}|u|^{k}u \\
&u(0) = \phi \in H^1(\mathbb{R}^n)
\end{aligned}
\right.
\end{equation}

The equation \eqref{inls} is a model for mainly two physical topics: the study of diluted Bose-Einstein condensation with two- and three-body interactions considered and the study of laser guiding in an axially non-uniform plasma channel. (See \citep{BelmonteBeitia200798}, \citep{YonggeunCho19}, \citep{Gill2000}, and \citep{TangShukla200776} for detail.)
INLS thus has many general forms, but we limit our attention to a relatively simple case of having one nonlinear power with a smooth (possibly except at origin) decaying multiplier, especially $|x|^{-b}$.

We start with introducing common terminologies in the studies of nonlinear Schr\"{o}dinger equations and their relatives.
The time interval $I$ in which the solution $u$ is defined is called its \emph{lifespan}.
For the sign parameter $\mu$, the equation \eqref{inls} is said to be \emph{defocusing} if $\mu = +1$ and \emph{focusing} if $\mu = -1$.
For the nonlinearity power parameter $k$, we consider together a value called the \emph{critical regularity} defined by $\gamma_c = n/2-(2-b)/k$. We note that both the equation \eqref{inls} and the solution's $\dot{H}_x^{\gamma_c}(\mathbb{R}^n)$ norm are preserved under the \emph{scaling} transformation
\[
u(t,x) ~\mapsto~ u_\lambda(t,x) := \lambda^{\frac{n-2\gamma_c}{2}}u(\lambda^2 t, \lambda x) \quad(\lambda>0).
\]

We also introduce two main quantities conserved throughout the lifespan of the solution $u \in C_t(I,H_x^1(\mathbb{R}^n))$. One is the \emph{mass}
\[
M[u(t)]
:= \int_{\mathbb{R}^n} \big|u(t,x)\big|^2 dx
\]
and the other is the \emph{energy}
\[
E[u(t)]
:=	\int_{\mathbb{R}^n} \frac{1}{2} \big|\nabla u(t,x)\big|^2 + \frac{\mu}{k+2}|x|^{-b}|u(t,x)|^{k+2} ~dx.
\]
Based on the fact that $M[u(t)]^{1-\gamma_c} E[u(t)]^{\gamma_c}$ is preserved under the scaling, we say that \eqref{inls} is \emph{mass-critical} if $\gamma_c = 0$, \emph{energy-critical} if $\gamma_c = 1$, and \emph{intercritical} if $0 < \gamma_c < 1$.

We lastly note one \emph{coercivity} property for the defocusing case.
It is the uniform boundedness of the solution's $\dot{H}_x^1(\mathbb{R}^n)$ norm in time, both above and below, in terms of the energy $E[\phi]$.
More precisely, every solution $u \in C_t(I,\dot{H}_x^1(\mathbb{R}^n))$ of \eqref{inls} with initial data $\phi$ satisfies the bounds
\[
c\min(E[\phi], E[\phi]^{c'}) \leq \norm{\dot{H}_x^1(\mathbb{R}^n)}{u(t)} \leq C\max(E[\phi], E[\phi]^{C'})
\]
for every $t\in I$, where $0<c<C$ and $0<c'<1<C'$ are absolute constants depending only on $n$ and $b$.
(As a side note, the focusing case also has a similar formulation for coercivity, but this is valid only when $E[\phi]$ is below a threshold.)

In this paper, our main interest lies in the case $\mu = +1$ and $k=(4-2b)/(n-2)$, and hence we have \emph{defocusing energy-critical INLS} as featured in the title. From now on, we limit our view to the defocusing energy-critical case unless stated otherwise.

When $b = 0$, \eqref{inls} turns into the usual nonlinear Schr\"{o}dinger equations (shortened as NLS), whose global existence and scattering results have been studied extensively for decades.
We mainly look at the defocusing energy-critical NLS as it is one of the two main inspirations for our result.
For the radial data, Bourgain \citep{Bourgain99} resolved $n=3$ and explicitly computed a bound of the solution's spacetime norms; it shows Knuth tower-type growth in terms of the solution's energy. Tao \citep{Tao05} extended the result to every $n \geq 3$, improving the bound to an exponential type also.
For the general data, Colliander, Keel, Staffilani, Takaoka, and Tao \citep{CollianderKSTT08} resolved $n=3$ and Ryckman and Visan \citep{RyckmanVisan07} resolved $n=4$, both of which gave Knuth tower type bounds. Though lacking quantitative bounds, Visan \citep{Visan07} resolved $n\geq 5$ using the frequency-localized interaction Morawetz inequality.

As a side note, we also look at some notable results for other parameter setups of NLS.
For the focusing energy-critical case, Kenig and Merle \citep{KenigMerle} proved the scattering result for $n=3,4,5$ with radial data, renowned for the nowadays widespread technique: the concentration--compactness principle.
For the focusing intercritical case, Dinh \citep{Dinh20u} provided a unified approach for the scattering result for every $n\geq 1$ with general data, which extends the earlier results of Dodson and Murphy (\citep{DodsonMurphyRad} and \citep{DodsonMurphy18}).
For the defocusing energy-supercritical case, Bulut \citep{Bulut23} resolved $(n, b, k)=(3, 0, 6)$ with radial data by inspiration from Tao \citep{Tao05}, while in contrast, Merle, Rapha\"{e}l, Rodnianski, and Szeftel \citep{MerleRRS22} studied the blow-up phenomena for some higher-dimensional triples $(n, b, k)=(5,0,9), (6,0,5), (8,0,3), (9,0,3)$.

We go back to $b>0$.
Recently, several studies for the scattering of INLS have emerged with varying parameters.
For the defocusing intercritical case, Dinh \citep{Dinh2019} resolved most of the intercritical cases with $n\geq 3$ (all intercritical cases with $n\geq 4$ in particular) by the arguments of Visciglia \citep{Visciglia09}.
For the focusing intercritical case, Cardoso, Farah, Guzman, and Murphy \citep{CardosoFGM20} resolved $n \geq 2$ with $0<b<\min(2, n/2)$ by the concentration compactness method, which extends the earlier results of Farah and Guzm\'{a}n (\citep{FarahGuzman17} and \citep{FarahGuzman20}), and Miao, Murphy, and Zhang \citep{MiaoMurphyZhang19}.
Dinh and Keraani \citep{DinhKeraani21} then gave a different scattering condition in terms of the ``potential energy'' and added the studies of other long-time dynamics such as blow-ups.
Campos and Cardoso \citep{CamposCardoso} provided a much different proof for $n\geq 3$ using the virial-Morawetz approach of Dodson and Murphy \citep{DodsonMurphyRad}.
For the focusing energy-critical case, Cho, Hong, and Lee (\citep{ChoHongLee20} and \citep{ChoLee21}) resolved $n=3$ with $0<b<3/2$ for radial data by the concentration compactness method, featuring weighted Sobolev spaces.
Guzman and Murphy \citep{GuzmanMurphy21} then extended the result to general data for $(n,b,k) = (3,1,2)$.

Meanwhile, Aloui and Tayachi (\citep{AlouiTayachiGWP} and \citep{AlouiTayachi}) have found some applications of Lorentz spaces for well-posedness results of INLS; local, global, and also the small-data scattering.
Both works follow the structure of Cazenave \citep{Cazenave03} but adopt Lorentz spaces in the $x$ variable to define Strichartz norms.
Many good properties of Lebesgue spaces (H\"{o}lder's inequality, Sobolev embedding, and Strichartz estimates) are also applicable to Lorentz spaces.

Regarding these backgrounds, we present the following two main theorems. In Theorem \ref{thrm:main}, we estimate a solution's spacetime norms for the defocusing energy-critical ILNS \eqref{inls} under the radial data assumption. The proof is based on the arguments of Tao \citep{Tao05} but adopts Lorentz spaces in the definition of solution spaces by inspiration from Aloui and Tayachi \citep{AlouiTayachi}.
We note that the spacetime norm estimate in Theorem \ref{thrm:main} never depends on the choice of $T_\pm$, and thus we can let $T_{\pm} \to \pm\infty$ for the global existence of the solution.
As a result, in conjunction with a local theory similar to \citep{AlouiTayachi},
we conclude in Theorem \ref{cor:scattering} that such a solution always scatters.

\begin{thrm} \label{thrm:main}
Let $n\geq 3$ be an integer and $b>0$ obey the following range restriction.
\[
\left\{
\begin{aligned}
&0<b<3/2 && \text{if } n=3, \\
&0<b<1 && \text{if } n=4, \\
&0<b<1/2 \text{\: or \:} 5/4<b<2 && \text{if } n=5, \\
&0<b<2 && \text{if } n\geq 6.
\end{aligned}
\right.
\]
Let $[T_-, T_+]$ be a compact interval, and let $u$ be the unique spherically symmetric solution of \eqref{inls} associated with the initial data $\phi \in H^1(\mathbb{R}^n)$ and lying in the solution space
\[
C_t([T_-, T_+], H_x^1(\mathbb{R}^n))
\cap L_t^{p_*}([T_-, T_+], W_x^{1;q_*,2}(\mathbb{R}^n))
\]
where $p_* = 2(n+2)/(n-2+2b)$ and $q_* = 2n(n+2)/(n^2+4-4b)$. Then we have
\[
\norm{L_t^{\frac{2(n+2)}{n-2}}([T_-, T_+], L_x^{\frac{2(n+2)}{n-2}}(\mathbb{R}^n))}{u}
\leq	C\exp(CE^C)
\]
where $E = E[\phi]$ is the solution's energy and $C$ are absolute constants depending only on $n$ and $b$.\\
In particular, if $n=3,4,5$ and $0<b<(6-n)/2$, the same estimate holds when we replace the inhomogeneous Sobolev-Lorentz spaces with homogeneous ones:
we can assume $\phi$ to only lie in $\dot{H}^1(\mathbb{R}^n)$,
in which case
the solution only lies in $C_t([T_-, T_+], \dot{H}_x^1(\mathbb{R}^n)) \cap L_t^{p_*}([T_-, T_+], \dot{W}_x^{1;q_*,2}(\mathbb{R}^n))$.
\end{thrm}

\begin{thrm} \label{cor:scattering}
Let $n$, $b$, $p_*$, $q_*$ and $u$ be defined as in Theorem \ref{thrm:main}. Then the solution $u$ is global and scatters in $\dot{H}_x^1(\mathbb{R}^n)$, in the sense that $u$ in fact belongs to
\[
C_t(\mathbb{R}, \dot{H}_x^1(\mathbb{R}^n))
\cap L_t^{p_*}(\mathbb{R}, \dot{W}_x^{1;q_*,2}(\mathbb{R}^n))
\]
and there exist profiles $U_+,U_-\in \dot{H}^1(\mathbb{R}^n)$ such that
\[
\lim_{t\to+\infty} \norm{\dot{H}_x^{1}}{u(t) - e^{it\Delta}U_+}
=\lim_{t\to-\infty} \norm{\dot{H}_x^{1}}{u(t) - e^{it\Delta}U_-}
= 0.
\]
\end{thrm}

\begin{remk}
In Theorem \ref{thrm:main}, the range conditions $n=3,4,5$ and $0<b<(6-n)/2$ correspond to the nonlinearity power $k>1$, high enough for the discussion of well-posedness in homogeneous Sobolev-Lorentz spaces.
Theorem \ref{cor:scattering} can be understood as a minimal result embracing both cases $k>1$ and $k\leq 1$. In fact, if the inhomogeneous Sobolev spaces are used in Theorem \ref{thrm:main}, then the scattering result in $H^1(\mathbb{R}^n)$ holds in Theorem \ref{cor:scattering}.
\end{remk}

Ideally, we expect Theorem \ref{thrm:main} to hold for every $n\geq 3$ and $0<b<\min(2,n/2)$ as this is the whole range where the local theory of the energy-critical case is discussed. However, the scattering problem remains open for two missing ranges $1\leq b < 2$ for $n=4$ and $1/2\leq b \leq 5/4$ for $n=5$. Both cases have the low dimension $n$ and the low nonlinearity power $k$ at the same time, making it challenging to apply the variant bootstrapping arguments by Tao \cite[Appendix]{Tao05}; we will see the detailed issues in Section \ref{Section:4}.

On the bright side, we still cover most of our target parameter range, especially all of $0<b<\min(2,n/2)$ when $n=3$ or $n\geq 6$.
Another interesting aspect is that we do not utilize the concentration compactness method of \citep{KenigMerle}, despite it being widespread for the scattering results of several dispersive equations nowadays.
Lastly, we can quantitatively bound the solution's size (in spacetime norms) using only the solution's energy and a few absolute parameters.


Throughout this paper, $C>0$ and $c>0$ stand for several large and small constants that depend only on the parameters $n$ and $b$ (unless stated otherwise). When several absolute constants exist in a single context, we may add prime symbols or subscripts for disambiguation.

To illustrate the arguments of Tao \cite[Section 3]{Tao05}, we start with a small parameter $\eta$, which preferably depends polynomially on the solution's energy $E$, and then divide the large time interval $[T_-, T_+]$ into many smaller subintervals $I_1, I_2, \cdots, I_J$ such that $L_t^{\frac{2(n+2)}{n-2}}(I_j, L_x^{\frac{2(n+2)}{n-2}})$ norm of the solution $u$ is comparable to $\eta^c$ for every $j = 1,\cdots,J$.
Later, it happens that one of the two scenarios happens for every $j$:
\begin{itemize}
\item	At least one of the two long-time approximants $u_\pm = e^{i(t-T\pm)\Delta}u(T_\pm)$ has significantly large $L_t^{\frac{2(n+2)}{n-2}}(I_j, L_x^{\frac{2(n+2)}{n-2}})$ norm, especially greater than $\eta^{C_1}$.
\item	There exists a point $x_j \in \mathbb{R}^n$ such that the solution forms a bubble of concentration on the spacetime slab $I_j \times B(x_j, C\eta^{-CC_0}|I_j|^{1/2})$. In particular, $L_x^2$ norm of $u(t)$ is bounded below by $c\eta^{CC_0}|I_j|^{1/2}$ uniformly in $t\in I_j$.
\end{itemize}

There are only a finite number of subintervals $I_j$ falling in the first scenario, called \emph{bad} (or \emph{exceptional}) subintervals in \citep{Tao05} and this paper. In particular, that number is polynomially bounded in $1/\eta$ and thus eventually negligible.
For all the other $I_j$, called \emph{good} (or \emph{unexceptional}) intervals, we can take $x_j = 0$ thanks to the radial data assumption. The Morawetz inequality of \citep{Bourgain99} has a variant for the defocusing energy-critical INLS, so we can extract a finite sequence of good subintervals rapidly concentrating to a single time. However, since the approximate local mass conservation law provides a lower bound for the local mass in each bubble of concentration, if $J$ is too large (beyond an exponential bound in $E$), then the summation of the local masses leads to a contradiction against the energy conservation law after H\"{o}lder's inequality.

As the distinction from \citep{Tao05}, we select the solution spaces differently from the usual $L^{p}$-based ones. Unlike NLS, INLS has a $L^{p,\infty}$ type multiplier $|x|^{-b}$. Since we work on the endpoint regularity of $H^1(\mathbb{R}^n)$ data assumption, it is necessary to use Lorentz spaces as in \citep{AlouiTayachi} or weighted Sobolev spaces as in \citep{ChoHongLee20}.

Lorentz spaces have two significant occurrences in this paper: a priori estimates and ``$h$-difference'' estimates.
A priori estimates are relatively straightforward due to the good properties of Lebesgue spaces shared with Lorentz spaces, and we only have to choose appropriate exponents.
The $h$-difference estimates refer to estimating the forms $f(x+h)-f(x)$, where $h \in \mathbb{R}^n$ is small and $f$ is any sufficiently smooth function in $x$ relevant to INLS, such as the multiplier $|x|^{-b}$ or the solution $u(t,x)$ at a fixed time $t$.
The real interpolation of Lorentz spaces and the Fundamental Theorem of Calculus are helpful for these estimates.

Most of the difficulties in this paper lie in a technical bootstrapping lemma, Lemma \ref{lem:3.2}.
The usual approach as in the short proof of \cite[Lemma 3.2]{Tao05} works only for high nonlinearity powers ($k>1$, i.e. $n<6-2b$), and this is where \cite[Appendix]{Tao05} comes in.
After differentiating \eqref{inls} in $x$, we find a way to bootstrap $\nabla u$ by Neumann series approximation.
To make use of the arguments' iterative nature, one special lemma (Lemma \ref{lem:4.1}) is necessary, which the technical conditions $n>4$ and $n(n-6)+4b>0$ for the remaining case $n \geq 6-2b$ arise from and explains the strange-looking loss of ranges in $b$ for $n=4,5$.

\section{Local Well-posedness and Other Basic Estimates} \label{Section:2}

\subsection{Definitions and Properties of Sobolev-Lorentz Spaces} \label{Subsection:2.1}

We start with some basic definitions for Lorentz spaces, Sobolev-Lorentz spaces, the Lorentz-type Strichartz norms, and the scattering behavior of a solution. For the differential operators like $|\nabla|^\gamma$ to make sense, we also assume that all the functions in this paper are tempered distributions in the $x$ variable.

\begin{defn}
Let $1<p<\infty, ~1\leq q\leq\infty$. A function $f:\mathbb{R}^n \to \mathbb{C}$ is said to be in the \emph{Lorentz space} $L^{p,q}(\mathbb{R}^n)$ if its \emph{Lorentz quasinorm}
\[
\norm{L^{p,q}(\mathbb{R}^n)}{f}
:=	\Big(\int_0^\infty \big(t^{-1/p}f^*(t)\big)^{q} \frac{dt}{t} \Big)^{1/q}
\]
is finite, where $f^*$ is the decreasing rearrangement of $f$.
\end{defn}

\begin{defn}
Let $\gamma\in\mathbb{R}, ~1<p<\infty, ~1\leq q\leq\infty$. A function $f:\mathbb{R}^n \to \mathbb{C}$ is said to be in the \emph{homogeneous Sobolev-Lorentz space} $\dot{W}^{\gamma;p,q}(\mathbb{R}^n)$ if
\[
\norm{\dot{W}^{\gamma;p,q}(\mathbb{R}^n)}{f}
:=	\big\Vert |\nabla|^\gamma f \big\Vert_{L^{p,q}(\mathbb{R}^n)} < \infty,
\]
or in the \emph{inhomogeneous Sobolev-Lorentz space} $W^{\gamma;p,q}(\mathbb{R}^n)$ if
\[
\norm{W^{\gamma;p,q}(\mathbb{R}^n)}{f}
:=	\big\Vert \langle\nabla\rangle^\gamma f \big\Vert_{L^{p,q}(\mathbb{R}^n)} < \infty
\]
where $\langle x \rangle := (1+|x|^2)^{1/2}$.
\end{defn}

\begin{defn}
Let $\gamma\in\mathbb{R}$, $2\leq p\leq\infty$, and $2\leq q<\infty$. A pair $(p,q)$ is said to be $\dot{H}^\gamma$-\emph{admissible} if $\dfrac{2}{p} + \dfrac{n}{q} = \dfrac{n}{2} - \gamma$, and the \emph{Strichartz space} $\dot{S}^{\gamma}(I)$ refers to the set of all functions $u:I\times\mathbb{R}^n\to\mathbb{C}$ whose \emph{Strichartz norms}, defined as below, are finite. (Usual modification if $p = \infty$.)
\begin{align*}
\left\Vert u \right\Vert_{\dot{S}^{\gamma}(I)}
&=	\sup_{\text{$L^2$-admissible }(p,q)} \left\Vert u \right\Vert_{L_t^p(I,\dot{W}_x^{\gamma;q,2}(\mathbb{R}^n))} \\
&=	\sup_{\text{$L^2$-admissible }(p,q)} \Big(\int_I \big\Vert |\nabla|^\gamma u(t) \big\Vert_{L_x^{q,2}(\mathbb{R}^n)}^p \,dt \Big)^{1/p}
\end{align*}
\end{defn}

\begin{remk}
Lorentz spaces $L^{p,q}(\mathbb{R}^n)$ are reflexive for every $1<p<\infty$ and $1< q < \infty$, and thus so are Sobolev-Lorentz spaces $\dot{W}^{\gamma;p,q}(\mathbb{R}^n)$.\footnote{See \cite[Chapter 2]{Lemarie02} or \cite[Theorem 1.4.16]{Grafakos}. Alternatively, we can combine the uniform convexity for the case $p\geq q$ (\cite[Section 3]{Halperin54}) with Milman-Pettis Theorem (\cite[Theorem 3.31]{Brezis}). The reflexivity fails at the endpoint cases $q=1, \infty$ since $(L^{p,\infty}(\mathbb{R}^n))^* \supsetneq L^{p',1}(\mathbb{R}^n)$; see \cite{Cwikel75}.}
When $p=q$, $\dot{W}^{\gamma;p,p}(\mathbb{R}^n)$ is equal to the usual Sobolev space $\dot{W}^{\gamma,p}(\mathbb{R}^n)$.
In the notation of norms, we often omit $\mathbb{R}^n$ when writing the solution space $L_t^p(I,\dot{W}_x^{\gamma;q,2}(\mathbb{R}^n))$, if the space variable $x$ runs through the entire $\mathbb{R}^n$ with Lebesgue measure and the context of the dimension $n$ is discernible.
We lastly note the following equivalent characterizations of Sobolev-Lorentz spaces for positive integers $\gamma$,
\begin{gather*}
c \leq \frac{ \norm{L^{p,q}(\mathbb{R}^n)}{\nabla^\gamma f} }{ \norm{\dot{W}^{\gamma;p,q}(\mathbb{R}^n)}{f} } \leq C
\text{\quad for every $0 \ne f \in \dot{W}^{\gamma;p,q}(\mathbb{R}^n)$}, \\
c \leq \frac{ \norm{L^{p,q}(\mathbb{R}^n)}{f} + \norm{L^{p,q}(\mathbb{R}^n)}{\nabla^\gamma f} }{ \norm{W^{\gamma;p,q}(\mathbb{R}^n)}{f} } \leq C
\text{\quad for every $0 \ne f \in W^{\gamma;p,q}(\mathbb{R}^n)$}
\end{gather*}
which can be derived from the real interpolation on H\"{o}rmander-Mikhlin multiplier theorem.
\end{remk}

We now look at several norm inequalities widely known for Lebesgue spaces. Apart from possibly worse proportionality constants, the inequalities work for Lorentz spaces too.
Aloui and Tayachi used these properties to establish the local well-posedness in \citep{AlouiTayachi} and the global well-posedness plus the small-data scattering in \citep{AlouiTayachiGWP}.

\begin{prop}[H\"{o}lder's inequality; {\cite[Proposition 2.3]{Lemarie02}, \cite[Theorem 3.4]{ONeil63}}]
Let $(X,\mu)$ be a $\sigma$-finite measure space, $1<p_1,p_2<\infty$, and $1\leq s_1,s_2\leq \infty$ such that
\[
1<p=\dfrac{1}{1/p_1+1/p_2}<\infty, \quad 1\leq s=\dfrac{1}{1/s_1+1/s_2} \leq\infty.
\]
Then for every $f \in L^{p_1,s_1}(X)$ and $g \in L^{p_2,s_2}(X)$, we have
\[
\norm{L^{p,s}(X)}{fg}
\leq C \norm{L^{p_1,s_1}(X)}{f} \norm{L^{p_2,s_2}(X)}{g}
\]
where $C$ is an absolute constant depending only on $p_1$, $p_2$, $s_1$, and $s_2$.
\end{prop}

\begin{prop}[Sobolev embeddings; {\cite[Theorem 2.4]{Lemarie02}}]
Let $0<\alpha<n$, $1<p<\tilde{p}<\infty$, and $1\leq s \leq\infty$ such that $-1/p + 1/\tilde{p} + \alpha/n = 0$. Then for every $f \in \dot{W}^{\alpha;p,s}(\mathbb{R}^n)$, we have
\[
\norm{L^{\tilde{p},s}(\mathbb{R}^n)}{f}
\leq C \norm{\dot{W}^{\alpha;p,s}(\mathbb{R}^n)}{f}
\]
where $C$ is an absolute constant depending only on $p$, $s$, $n$, and $\alpha$.
\end{prop}

\begin{prop}[Strichartz estimates; {\citep{KeelTao}}]~
\begin{itemize}
\item[(i)]
Let $2<p<\infty$. Then for every $f \in L^{p'}(\mathbb{R}^n)$ and $t\ne 0$, we have the dispersive estimate
\[
\norm{L_x^{p,2}(\mathbb{R}^n)}{e^{it\Delta}f}
\leq C |t|^{-n(\frac{1}{2} - \frac{1}{p})} \norm{L^{p',2}(\mathbb{R}^n)}{f}
\]
where $1/p' + 1/p = 1$.
\item[(ii)]
Let $(p,q) \in [2,\infty]\times[2,\infty)$ be an $L^2$-admissible pair. Then for every $f \in L^{2}(\mathbb{R}^n)$, we have
\[
\norm{L_t^p(\mathbb{R}, L_x^{q,2}(\mathbb{R}^n))}{e^{it\Delta}f}
\leq C \norm{L^{2}(\mathbb{R}^n)}{f}
\]
where $C$ is an absolute constant depending only on $p$, $q$, and $n$.
If $(\tilde{p},\tilde{q})$ is another $L^2$-admissible pair and $F \in L_t^{p'}(\mathbb{R}, L_x^{q',2}(\mathbb{R}^n))$, then we also have
\[
\norm{L_t^{\tilde{p}}(\mathbb{R}, L_x^{\tilde{q},2}(\mathbb{R}^n))}{\int_{0}^{t} e^{i(t-s)\Delta}F(s) ds}
\leq C \norm{L_t^{p'}(\mathbb{R}, L_x^{q',2}(\mathbb{R}^n))}{F}
\]
where $C$ is an absolute constant depending only on $p$, $\tilde{p}$, $q$, $\tilde{q}$, and $n$.
\end{itemize}
\end{prop}

\begin{remk}
(i) comes from the real interpolation of $L^2 \to L^2$ and $L^1 \to L^\infty$ operator norms of $e^{it\Delta}$ in the first page of \citep{KeelTao}. For details, see the off-diagonal Marcinkiewicz interpolation theorem \cite[Theorem 1.4.19]{Grafakos}.
(ii) is a restatement of \cite[Theorem 10.1]{KeelTao}.
\end{remk}

\begin{prop}[{\cite[Proposition 6]{AlouiTayachi}}] \label{prop:2.4-compact-dct}
Let $\gamma \geq 0$ and $K$ be a compact subset of $H^\gamma(\mathbb{R}^n)$. Then for every $L^2$-admissible pair $(p,q)$ with $p \ne \infty$, we have
\[
\lim_{T\to+0} \sup_{f\in K} \norm{L_t^p([0,T], W_x^{\gamma;q,2}(\mathbb{R}^n))}{e^{it\Delta}f} = 0.
\]
\end{prop}

\begin{remk}
This proposition also has an analog for homogeneous Sobolev-Lorentz spaces.
\end{remk}

\subsection{Some Local Well-posedness Results} \label{Subsection:2.2}

In this subsection, it would be convenient to abbreviate some critical exponents in the Lebesgue- or Lorentz-type function spaces. For $n\geq 3$ and $0\leq \beta<\min(2, n/2)$, we denote
\[
p_\beta = \frac{2(n+2)}{n-2+2\beta}, \quad
q_\beta = \frac{2n(n+2)}{n^2+4-4\beta}, \quad
\tilde{q}_\beta = \frac{2n(n+2)}{n^2-2n-4\beta}.
\]
(We only consider $0\leq \beta<3/4$ for $\tilde{q}_\beta$ when $n=3$.)
Some notable properties are as follows.
\begin{itemize}
\item $(p_\beta, q_\beta)$ is $L^2$-admissible and $(p_\beta, \tilde{q}_\beta)$ is $\dot{H}^1$-admissible.
\item $p_0 = \tilde{q}_{0} = \dfrac{2(n+2)}{n-2}$ and
$\Big(\dfrac{1}{2}, \dfrac{n+2}{2n}\Big) = \Big(0, \dfrac{b}{n}\Big) + k\Big(\dfrac{1}{p_0}, \dfrac{1}{\tilde{q}_{0}}\Big) + \Big(\dfrac{1}{p_b}, \dfrac{1}{q_b}\Big)$.
\item Sobolev embedding $\dot{W}^{1;q_\beta,2}(\mathbb{R}^n) \hookrightarrow L^{\tilde{q}_\beta,2}(\mathbb{R}^n)$ holds.
\end{itemize}
We are now ready to discuss the local well-posedness of INLS \eqref{inls} in the function space
\[
C_t([0, T), H_x^1(\mathbb{R}^n)) \cap L_t^{p_b}([0, T), W_x^{1;q_b,2}(\mathbb{R}^n)).
\]
Our arguments have a minor alteration from \citep{AlouiTayachi};
we use two \emph{controlling spaces} together
\[
L_t^{p_0}([0, T], W_x^{1;q_0,2}(\mathbb{R}^n)) \cap L_t^{p_b}([0, T], W_x^{1;q_b,2}(\mathbb{R}^n))
\]
rather than one controlling space $L_t^{k+2}([0, T], W_x^{1;\frac{n(k+2)}{n-b+k},2}(\mathbb{R}^n))$ as in \citep{AlouiTayachi}.
Under this alternative setup, the coverage of $b$ for $n=3$ improves from $0<b<1$ to $0<b<3/2$.
The ideas of proof are provided separately when necessary.

\begin{thrm}[Local well-posedness in $H^1$; {\cite[Theorems 1.2 and 1.4]{AlouiTayachi}}] \label{thrm:2.5-well-posedness}
Let $n\geq 3$, $0<b<\min(2,n/2)$. Then every initial data $\phi \in H^1(\mathbb{R}^n)$ admits a unique solution $u$ of INLS \eqref{inls} with the maximal lifespan $T_*(\phi)>0$ in the solution space
\[
C_t([0, T_*(\phi)), H_x^1(\mathbb{R}^n)) \cap L_t^{p_b}([0, T_*(\phi)), W_x^{1;q_b,2}(\mathbb{R}^n)).
\]
In addition, we have the following properties.

\begin{itemize}
\item[(i)]
$u$ is unique in $L_t^{p_0}([0, T], W_x^{1;q_0,2}(\mathbb{R}^n)) \cap L_t^{p_b}([0, T], W_x^{1;q_b,2}(\mathbb{R}^n))$ for $0<T<T_*(\phi)$.
\item[(ii)]
$u \in L_t^p([0, T], W_x^{1;q,2}(\mathbb{R}^n))$ for $0<T<T_*(\phi)$ and every $L^2$-admissible pair $(p,q)$.
\item[(iii)]
If $\norm{\dot{H}_x^1}{\phi}$ is small enough, then $T_*(\phi) = \infty$ and $u$ lies in $L_t^p([0, \infty), W_x^{1;q,2}(\mathbb{R}^n))$ for every $L^2$-admissible pair $(p,q)$.
Moreover, $u$ scatters in $H_x^1(\mathbb{R}^n)$ forward in time, meaning that $\lim\limits_{t\to\infty} e^{-it\Delta}u(t)$ exists in the sense of $H_x^1(\mathbb{R}^n)$.
\item[(iv)]
If $k \geq 1$, then for every $0<T<T_*(\phi)$, there is $\delta_0 > 0$ such that if $\psi \in H_x^1(\mathbb{R}^n)$ with $\norm{H_x^1}{\phi-\psi} < \delta_0$, we have $T_*(\psi) > T$ and for every $L^2$-admissible pair $(p,q)$, the corresponding maximal solution $v$ satisfies
\[
\norm{L_t^{p}([0, T], W_x^{1;q,2}(\mathbb{R}^n))} {u - v}
\leq C \norm{H_x^1(\mathbb{R}^n)} {\phi - \psi}.
\]
\item[(v)]
If $k<1$, and if $\phi_j \to \phi$ in $H_x^1(\mathbb{R}^n)$,
then for every $0<T<T_*(\phi)$,
we have $u_j \to u$ in $C_t([0, T], H_x^1(\mathbb{R}^n))$,
where $u_j$ is the solution of INLS \eqref{inls} with initial data $\phi_j$.
\end{itemize}
\end{thrm}

\begin{proof}
We sketch the proof for the inhomogeneous Sobolev space setups, which works whenever $k>0$.
Once we determine the solution, we can show the properties (i)--(iv) with a priori estimates.

Consider two parameters $T,M>0$. We start with defining the set
\begin{align*}
X = X(T,M) := \big\{ & u \in L_t^{p_0}([0, T], W_x^{1;q_0,2}(\mathbb{R}^n)) \cap L_t^{p_b}([0, T], W_x^{1;q_b,2}(\mathbb{R}^n)) : \\
& \max_{\beta=0,b} \norm{L_t^{p_\beta}([0, T], \dot{W}_x^{1;q_\beta,2}(\mathbb{R}^n)) }{u} \leq M
\big\}
\end{align*}
equipped with the metric
\[
d(u,v) := \max_{\beta=0,b} \norm{L_t^{p_\beta}([0, T], L_x^{q_\beta,2}(\mathbb{R}^n)) }{u-v}.
\]
Note how changing the choice of controlling spaces affects the definition of $X$ and $d$.

We claim that the metric space $(X, d)$ is complete. We note that the function space
\[
L_t^{p_0}([0, T], \dot{W}_x^{1;q_0,2}(\mathbb{R}^n)) \cap
L_t^{p_b}([0, T], \dot{W}_x^{1;q_b,2}(\mathbb{R}^n))
\]
is a reflexive Banach space, since $L^{p,2}(\mathbb{R}^n)$ is reflexive for every $1<p<\infty$ and also the intersection of two compatible reflexive Banach spaces is reflexive.\footnote{See \cite[Chapter 3, pp.174--175]{BennettColin88} for the intersection of two compatible reflexive Banach spaces.}
A closed unit ball in this space is thus sequentially compact in weak topology (\cite[Theorems 3.7, 3.18]{Brezis}), and we obtain the claim.

To establish a priori estimates,
we define the Duhamel operator $\mathscr{D}_{\phi}(u)$ as below, where $u \in X$ and $\phi \in H_x^1(\mathbb{R}^n)$ stand for a solution and its initial data respectively.
\[
\mathscr{D}_{\phi}(u) = e^{it\Delta}\phi - i\mu\int_0^t {e^{i(t-s)\Delta}(|x|^{-b}|u|^ku)(s)} ds
\]
Let $(p,q)$ be any $L^2$-admissible pair. Then we have the following nonlinearity estimate.
\begin{equation} \label{nonl-pt-est}
\begin{aligned}
&	\norm{L_t^p([0, T], \dot{W}_x^{1;q,2})}{\int_0^t {e^{i(t-s)\Delta}(|x|^{-b}|u|^ku)(s)} ds}
\\
\leq\;&	C \norm{L_t^{2}([0, T], \dot{W}_x^{1;\frac{2n}{n+2},2})}{|x|^{-b}|u|^ku}
\\
\leq\;&	C \norm{L_x^{\frac{n}{b},\infty}}{|x|^{-b}}
\norm{L_t^{2}([0, T], \dot{W}_x^{1;\frac{2n}{n+2-2b},2})}{|u|^ku}
\\
&	+ C\norm{L_x^{\frac{n}{b+1},\infty}}{|x|^{-(b+1)}}
\norm{L_t^{2}([0, T], L_x^{\frac{2n}{n-2b},2})}{|u|^ku}
\\
\leq\;&	C \norm{L_t^{2}([0, T], \dot{W}_x^{1;\frac{2n}{n+2-2b},2})}{|u|^ku}
\\
\leq\;&	C \norm{L_t^{p_0}([0, T], L_x^{{\tilde{q}_0},2})}{u}^k
\norm{L_t^{p_b}([0, T], \dot{W}_x^{1;q_b,2})}{u}
\end{aligned}
\end{equation}
This leads to the following slightly altered a priori estimate in $\dot{S}^1([0, T])$ level,
\begin{equation} \label{alt-apriori-S1}
\begin{aligned}
&	\norm{L_t^p([0, T], \dot{W}_x^{1;q,2})}{\mathscr{D}_{\phi}{u}} \\
\leq\,&	\norm{L_t^p([0, T], \dot{W}_x^{1;q,2})}{e^{it\Delta}\phi}
+ C \norm{L_t^{p_0}([0, T], \dot{W}_x^{1;q_0,2})}{u}^k
\norm{L_t^{p_b}([0, T], \dot{W}_x^{1;q_b,2})}{u}
\end{aligned}
\end{equation}
and similarly we also have the following estimate in $\dot{S}^0([0, T])$ level.
\begin{equation} \label{alt-apriori-S0}
\begin{aligned}
&	\norm{L_t^p([0, T], L_x^{q,2})}{\mathscr{D}_\phi{u} - \mathscr{D}_\psi{v}} \\
\leq\;&	C \norm{L_x^2}{\phi - \psi}
+	C \big(\! \norm{L_t^{p_0}([0, T], \dot{W}_x^{1;q_0,2})}{u}^k + \norm{L_t^{p_0}([0, T], \dot{W}_x^{1;q_0,2})}{v}^k \!\big)
\norm{L_t^{p_b}([0, T], L_x^{q_b,2})}{u-v}
\end{aligned}
\end{equation}
Let $0 < \epsilon < \epsilon_0 = \epsilon_0(n,b)$ be small enough. After choosing $M = 2\epsilon$ and $T>0$ such that
\begin{equation} \label{lwp-T-determine}
\max_{\beta=0,b}
\norm{L_t^{p_\beta}([0, T], \dot{W}_x^{1;q_\beta,2}(\mathbb{R}^n))}{e^{it\Delta}\phi}
< \epsilon,
\end{equation}
the operator $\mathscr{D}_\phi: X \to X$ becomes a contraction mapping so we can apply Banach fixed point theorem. This proves the existence and uniqueness of the solution $u$ in $X$.

To show the five additional properties,
we can modify the arguments of \cite[pp. 5423--5426]{AlouiTayachi} for (i)--(iii) and \cite[Section 5]{AlouiTayachi} for (iv)--(v), keeping track of our choice of controlling spaces.
We only look at the proof of (iv) as an example for brevity.
Let $\epsilon>0$ be small, $K \subset H_x^1(\mathbb{R}^n)$ be compact, $\phi \in K$, and $\psi$ be close to $\phi$ in $H_x^{1}(\mathbb{R}^n)$ as assumed. Then by Proposition \ref{prop:2.4-compact-dct} applied to $K$, there exists $T>0$ satisfying the following smallness conditions, where $u$ and $v$ are the solutions from the initial data $\phi$ and $\psi$ respectively. Note our change of controlling spaces in effect.
\[
\max_{w=u,v;\;\beta=0,b}\norm{L_t^{p_\beta}([0, T], \dot{W}_x^{1;q_\beta,2}(\mathbb{R}^n))}{w} < 2\epsilon
\]
From the following pointwise inequalities, where $k\geq 1$ is required for the second one,
\begin{align*}
\big| |u|^ku - |v|^kv \big|
&	\leq C (|u|^k + |v|^k) |u-v|, \\
\big| \nabla(|u|^ku - |v|^kv) \big|
&	\leq C |u|^k |\nabla(u-v)| + C (|u|^{k-1} + |v|^{k-1}) |u-v| |\nabla v|
\end{align*}
we can deduce the following a priori estimate of $u-v$ for every $L^2$-admissible pair $(p,q)$ and both regularities $\gamma=0,1$, similarly to how we derived the nonlinearity estimate \eqref{nonl-pt-est}.
\begin{align*}
\norm{L_t^{p}([0, T], \dot{W}_x^{\gamma;q,2})} {u - v}
\leq	C \norm{\dot{H}_x^\gamma} {\phi - \psi}
+ C \epsilon^k
\max_{\beta=0,b} \norm{L_t^{p_\beta}([0, T], \dot{W}_x^{\gamma;q_\beta,2})} {u-v}
\end{align*}
The claims that $T_*(\phi), T_*(\psi) > T$ and the a priori estimate of $u-v$ holds for every $0<T<T_*(\phi)$ can be proved with a similar argument to \cite[Section 5, Proof of (i)]{AlouiTayachi}.
\end{proof}

\begin{remk}
When $k> 1$ and the initial data lies in $\dot{H}_x^1(\mathbb{R}^n)$, we also have the local well-posedness results in homogeneous solution spaces, in the sense that we can change all the function spaces in Theorem \ref{thrm:2.5-well-posedness} into homogeneous forms such as $L_t^p([0, T], \dot{W}_x^{1;q,2}(\mathbb{R}^n))$ and $\dot{H}_x^{1}(\mathbb{R}^n)$.
In this case, we define $X = X(T,M)$ in the proof as a closed ball of radius $M$ in the function space
\[
L_t^{p_0}([0, T], \dot{W}_x^{1;q_0,2}(\mathbb{R}^n)) \cap L_t^{p_b}([0, T], \dot{W}_x^{1;q_b,2}(\mathbb{R}^n))
\]
with its natural metric.
This time, we do not sacrifice regularity for defining the metric, so the completeness of $X$ follows relatively easily.
After that, the procedure of showing the solution's existence, uniqueness, and the four additional properties (i)--(iv) only needs a priori estimates in $\dot{S}^1(\mathbb{R}^n)$ level, such as \eqref{alt-apriori-S1} but not \eqref{alt-apriori-S0}.
See \cite[Theorem 2.5]{KenigMerle} for the idea.
\end{remk}

\subsection{Morawetz Inequality for INLS} \label{Subsection:2.3}
This subsection extends Morawetz inequality introduced by Bourgain \citep{Bourgain99}. It initially appears for NLS with $n=3$ but is extended to every dimension $n\geq 3$ by Tao \citep{Tao05}. In fact, we can further modify the inequality to fit into INLS with every $n\geq 3$ and $0<b<\min(2,n/2)$.
We start with the following variant of Morawetz identity.

\begin{lem}[Morawetz identity for INLS]
Let $V, a : \mathbb{R}^n \to \mathbb{R}$ be two non-negative smooth functions with compact support.
Suppose that $u:I\times \mathbb{R}^n \to \mathbb{C}$ is a solution of the equation
\[
iu_t + \Delta u = V(x)|u|^{k}u
\]
and define
\[
Z(t) = 2\,\mathrm{Im} \int_{\mathbb{R}^n} \bar{u}\nabla u \cdot \nabla a ~dx.
\]
Then the following identity holds.
\begin{align*}
\partial_t Z(t)
&=	4 \sum_{j,k=1}^{n} \int_{\mathbb{R}^n} \mathrm{Re}(\overline{\partial_{x_j}u} {\partial_{x_k}u}) \partial_{x_j}\partial_{x_k}a ~dx
- \int_{\mathbb{R}^n} |u|^2 \Delta^2 a ~dx \\
&\quad	+ \frac{2k}{k+2} \int_{\mathbb{R}^n} V|u|^{k+2} \Delta a ~dx
- \frac{4}{k+2} \int_{\mathbb{R}^n} |u|^{k+2} \,\nabla V \cdot \nabla a ~dx
\end{align*}
\end{lem}

To construct an inequality from this identity, choose $V(x)=|x|^{-b}$ and $a(x) = (|x|^2+\epsilon^2)^{1/2}\chi(x)$, where $\epsilon>0$ is small and $\chi \in C_0^\infty(\mathbb{R}^n)$ is a radially decreasing bump function with $\chi(x) = 1$ for $|x|\leq 1$ and $\chi(x) = 0$ for $|x|\geq 2$.
We observe the convexity of $a$ on $|x|\leq 1$, the non-negativity of the following three functions on $|x|\leq 1$,
\begin{align*}
x\cdot \nabla a &= \frac{|x|^2}{(|x|^2+\epsilon^2)^{1/2}} \\
\Delta a &= \frac{n-1}{(|x|^2+\epsilon^2)^{1/2}} + \frac{\epsilon^2}{(|x|^2+\epsilon^2)^{3/2}} \\
-\Delta^2 a &= \frac{(n-1)(n-3)}{(|x|^2+\epsilon^2)^{3/2}} + \frac{6(n-3)\epsilon^2}{(|x|^2+\epsilon^2)^{5/2}} + \frac{15\epsilon^4}{(|x|^2+\epsilon^2)^{7/2}}
\end{align*}
and the uniform boundedness of $a$ and its derivatives on $1\leq|x|\leq2$ with respect to $\epsilon$. Ignoring some positive terms, we see that
\begin{equation*}
\partial_t Z(t)
\geq	c \int_{|x|\leq 1} \frac{|u|^{k+2}}{|x|^{b}(|x|^2+\epsilon^2)^{1/2}} dx - CE^C.
\end{equation*}
\noindent
Integrating this in time, using the scaling invariance, and sending $\epsilon \to 0$, we obtain the following variant of Morawetz inequality.

\begin{lem}[Morawetz inequality for INLS, localized in space]
\label{lem:2.7-inls-morawetz}
Let $u:I\times \mathbb{R}^n \to \mathbb{C}$ be a solution for the equation \eqref{inls}.
Then we have the inequality
\begin{equation*}
\int_I \int_{|x|\leq A|I|^{1/2}} \frac{|u|^{2(n-b)/(n-2)}}{|x|^{b+1}} dx\, dt
\leq CE^C A |I|^{1/2}
\end{equation*}
for every interval $I \subset \mathbb{R}$ and $A \geq 1$.
\end{lem}

\subsection{Local Mass Conservation} \label{Subsection:2.4}
Lastly, we also define the \emph{local mass} of a solution as in Tao \citep{Tao05};
\[
\mathrm{Mass}(u(t); B(x_0, R))
:=	\Big(\int_{\mathbb{R}^n} \chi^2\Big(\frac{x-x_0}{R}\Big) |u(t,x)|^2~dx \Big)^{1/2}
\]
where $B(x_0, R)$ stands for the ball $\{x\in\mathbb{R}^n : |x-x_0|<R\}$ and the bump function $\chi$ is same as before.
Albeit similar to the ``global'' $L^2$ mass, the integrand is multiplied by a smooth cutoff function around a ball with the supposedly large radius $R>0$.
After some computation, for every radius $R$, we can see that
\begin{align}
|\mathrm{Mass}(u(t); B(x_0, R))| &\leq CE^{1/2}R, \label{2.lm-i} \\
|\partial_t \mathrm{Mass}(u(t); B(x_0, R))| &\leq \frac{CE^{1/2}}{R}. \label{2.lm-ii}
\end{align}
The property \eqref{2.lm-ii} is called the \emph{approximate conservation law of local masses} since the change of a local mass becomes slighter as $R$ grows. Whether the NLS equation is standard or inhomogeneous does not matter in verifying both properties.

\section{Proof of Theorem \ref{thrm:main}} \label{Section:3}

\subsection{The Main Road Map} \label{Subsection:3.1}

We will follow the idea of Tao \citep{Tao05} but use Lorentz spaces for some necessary places.
We start with fixing $E,\, [T_-, T_+],\, u$.
Since the scattering result is already established when the energy $E$ is small enough, we also assume $E \geq c > 0$.
We aim to show the existence of large absolute constants $C_0$, $C_1$, $C_2$, $C_3$, and $C_4$ depending only on $n$ and $b$ such that
\[
\int_{T_-}^{T_+}\int_{\mathbb{R}^n} |u|^{\frac{2(n+2)}{n-2}} dx\,dt
\leq	\overline{C} \exp(\overline{C}E^{\overline{C}})
\]
where $\overline{C}$ is a large absolute constant depending only on $n$, $b$, and $C_k$.
The constants $C_k$ will follow the order $C_0 \ll C_1 \ll C_4/C_2$ and $C^{C_2}\ll C_3$.

We introduce a small parameter $\eta = 1/(C_3 E^{C_2})$ and partition $[T_-, T_+]$ into disjoint subintervals $I_1,\, I_2,\, \cdots,\, I_J$ such that
\begin{equation} \label{partitioned}
\eta \leq \int_{I_j}\int_{\mathbb{R}^n} |u|^{\frac{2(n+2)}{n-2}} dx\,dt
\leq 2\eta.
\end{equation}
Our aim changes to estimating $J$ uniformly in time $T_\pm$, especially in the exponential form $J\leq \overline{C} \exp(\overline{C}E^{\overline{C}})$.
We assume $J \gg 1$ as the case $J$ is too small can be eliminated by taking a longer lifespan $[T_+,T_-]$ or using the small-data scattering theory.
We also note that $J$ is a finite number, given the compactness of $[T_+,T_-]$ and the local theory.

We start with a bootstrapping lemma that provides the $\dot{S}^1(I_j)$ boundedness uniformly in $j$. It holds for the full range $n\geq 3$ and $0<b<\min(2,n/2)$.

\begin{lem} \label{lem:3.1}
For every subinterval $I_j$, we have
\[
\norm{\dot{S}^{1}(I_j)}{u} \leq CE^C.
\]
\end{lem}
\begin{proof}
Fix $t_0\in I_j$. For $\eta$ small enough, by the nonlinearity estimate \eqref{nonl-pt-est}, we see that
\begin{align*}
\norm{\dot{S}^1(I_j)} {u}
&\leq	\norm{\dot{S}^1(I_j)} {e^{it\Delta}u(t_0)}
+ \norm{\dot{S}^1(I_j)}
{\int_{t_0}^{t} e^{i(t-s)\Delta}(|x|^{-b}|u|^ku)(s) \,ds} \\
&\leq	C \norm{\dot{H}_x^1}{u(t_0)}
+ C \norm{L_t^{\frac{2(n+2)}{n-2}}(I_j, L_x^{\frac{2(n+2)}{n-2}})} {u}^{\frac{2(2-b)}{n-2}}
\norm{L_t^{\frac{2(n+2)}{n-2+2b}}(I_j, \dot{W}_x^{1;\frac{2n(n+2)}{n^2+4-4b},2})} {u} \\
&\leq	C E^C + C \eta^{\frac{2-b}{n+2}} \norm{\dot{S}^1(I_j)} {u}
\end{align*}
and the result follows from simple bootstrapping.
\end{proof}

There is another necessary yet technical bootstrapping lemma. Given a general time interval $[t_1, t_2] \subset [T_-, T_+]$, if the data size is comparable to $\eta$ in $L_{t,x}^{2(n+2)/(n-2)}$ norm, the lemma ensures a lower bound, polynomially depending on $\eta$, for the size of the linearly evolved approximants at both endpoints $t_1$ and $t_2$.

\begin{lem} \label{lem:3.2}
Let $[t_1, t_2] \subset [T_-, T_+]$ be an interval such that
\begin{equation} \label{lem3.2-assume}
\frac{\eta}{2}
\leq \int_{t_1}^{t_2}\int_{\mathbb{R}^n} |u|^{\frac{2(n+2)}{n-2}} dx\,dt
\leq 2\eta.
\end{equation}
If we define $u_m(t,x) = e^{i(t-t_m)\Delta}u(t_m)$ ($m=1,2$), then for both $m=1,2$, we have
\[
\int_{t_1}^{t_2}\int_{\mathbb{R}^n} |u_m|^{\frac{2(n+2)}{n-2}} dx\,dt \geq c\eta^C.
\]
\end{lem}

Unfortunately, this lemma is where our arguments no longer retain the ideal parameter range $n\geq 3$ and $0<b<\min(2,n/2)$. We first provide a straightforward proof that holds for every lower-dimensional case $n < 6-2b$. The remaining higher-dimensional cases $n \geq 6-2b$ are dealt with in Section \ref{Section:4}, which follows the idea of \cite[Appendix]{Tao05}. Although not all higher-dimensional cases are coverable, most still are.

\begin{proof}[Proof of Lemma \ref{lem:3.2} for lower dimensions]~

Duhamel's formula on $u - u_m$ and the nonlinearity estimate \eqref{nonl-pt-est} from time $t_m$ yields
\[
\norm{\dot{S}^1([t_1,t_2])}{u - u_m} \leq CE^C \eta^{\frac{2-b}{n+2}}
\]
and hence by the Sobolev embedding $\dot{W}_x^{1;\frac{2n(n+2)}{n^2+4},2}(\mathbb{R}^n) \hookrightarrow L_x^\frac{2(n+2)}{n-2}(\mathbb{R}^n)$, we have
\[
\norm{L_t^{\frac{2(n+2)}{n-2}}([t_1,t_2];L_x^{\frac{2(n+2)}{n-2}})} {u - u_m}
\leq CE^C \eta^{\frac{2-b}{n+2}}.
\]
The result is then obtained in view of \eqref{partitioned} and the triangle inequality. Note that this relatively standard bootstrapping approach is valid only when
\[
\frac{n-2}{2(n+2)} < \frac{2-b}{n+2},
\]
which coincides with the lower-dimensional condition $n<6-2b$.
\end{proof}

Let us consider two endpoint linear evolutions $u_-,\, u_+: [T_-, T_+]\times\mathbb{R}^n \to \mathbb{C}$ defined by $u_\pm(t) = e^{i(t-T_\pm)\Delta}u(T_\pm)$, which supposedly approximate the long-time behavior of the solution $u$.
We categorize the time subintervals $I_j$ into two groups: $I_j$ is said to be either a \emph{bad} (or \emph{exceptional}) subinterval if
\[
\int_{I_j}\int_{\mathbb{R}^n} |u_\pm|^{\frac{2(n+2)}{n-2}} dx\,dt
> \eta^{C_1}
\]
for at least one of the signs $\pm$, or a \emph{good} (or \emph{unexceptional}) subinterval otherwise. We note that, from the simple Strichartz estimate
\[
\int_{T_-}^{T_+}\int_{\mathbb{R}^n} |u_\pm|^{\frac{2(n+2)}{n-2}} dx\,dt
\leq	CE^C,
\]
there are at most $CE^C\eta^{-C_1}$ exceptional subintervals, meaning we only have to focus on estimating the number of good subintervals.

We now look at a proposition describing a lower bound of the solution's local mass in a good subinterval. It is used to establish a certain relationship for the lengths of good subintervals, which later helps the estimation of $J$. We defer its proof to Subsection \ref{Subsection:3.2}.

\begin{prop} \label{prop:3.3-locmass}
Let $I_j$ be a good subinterval. Then there exists $x_j\in\mathbb{R}^n$ such that
\[
\mathrm{Mass}(u(t); B(x_j, C\eta^{-CC_0}|I_j|^{1/2})) \geq c\eta^{CC_0} |I_j|^{1/2}
\]
for every $t\in I_j$.
\end{prop}

The radiality condition becomes essential afterward.
Adjusting the absolute constants $C$ (but not $C_0$ and $\eta$) in the proposition if necessary, we can reduce the cases to $x_j = 0$ for every good subinterval $I_j$ by the radiality condition.
The idea is that if $|x_j|$ is not adequately bounded, then the data becomes too large in $L_x^{\frac{2n}{n-2}}(\mathbb{R}^n)$ on a disjoint union of the ball $B(x_j, C\eta^{-CC_0}|I_j|^{1/2})$ and its selected rotations, violating the energy conservation. See \cite[Corollary~3.5]{Tao05} for the proof.

\begin{cor} \label{cor:3.4-locmass-at-0}
Let $I_j$ be a good subinterval and $u$ be spherically symmetric. Then
\[
\mathrm{Mass}(u(t); B(0, C\eta^{-CC_0}|I_j|^{1/2})) \geq c\eta^{CC_0} |I_j|^{1/2}
\]
for every $t\in I_j$.
\end{cor}

Combining Corollary \ref{cor:3.4-locmass-at-0} and Lemma \ref{lem:2.7-inls-morawetz}, we deduce the following relation for the lengths of good subintervals, similar to \cite[Corollary~3.6]{Tao05}. The proof is provided separately as the Morawetz inequality for INLS differs slightly from the case of standard NLS.

\begin{cor} \label{cor:3.5}
Let the solution $u$ be spherically symmetric. Then for any interval $I\subset[T_-, T_+]$, we have
\[
\sum_{\substack{1\leq j\leq J\\[.25ex] I_j\subset I \mathrm{\:good}}} |I_j|^{1/2}
\leq C \eta^{-CC_0} |I|^{1/2}.
\]
\end{cor}

\begin{proof}
Let $R_0 = C\eta^{-CC_0}|I_j|^{1/2}$.
Taking H\"{o}lder's inequality on Corollary \ref{cor:3.4-locmass-at-0}, we have the following lower bounds for every good subinterval $I_j$ and $R \geq R_0$. We especially consider the case $R = C\eta^{-CC_0}|I|^{1/2}$.
\begin{align*}
&\quad	\int_{I_j} \int_{|x|\leq R} \frac{|u|^{2(n-b)/(n-2)}}{|x|^{b+1}} dx\, dt \\
&\geq	|I_j| R_0^{-(b+1)} \inf_{t\in I_j} \int_{|x|\leq R_0} |u(t)|^{2(n-b)/(n-2)} dx \\
&\geq	|I_j| R_0^{-(b+1)} (CR_0^n)^{-(2-b)/(n-2)} \inf_{t\in I_j} \Big( \int_{|x|\leq R_0} |u(t)|^{2} dx \Big)^{\frac{n-b}{n-2}} \\
&\geq	c \eta^{CC_0} |I_j|^{1/2}
\end{align*}
Meanwhile, Lemma \ref{lem:2.7-inls-morawetz} yields the upper bound
\[
\int_{I} \int_{|x|\leq R} \frac{|u|^{2(n-b)/(n-2)}}{|x|^{b+1}} dx\, dt \leq CE^C \eta^{-CC_0} |I|^{1/2}.
\]
Combining the upper bound and the summation of lower bounds yields the claim.
\end{proof}

Taking a specific type of $I$ in Corollary \ref{cor:3.5}, a union of adjacent good intervals, leads to the following.
See \cite[Corollary~3.7]{Tao05} for the proof.

\begin{cor} \label{cor:3.6}
Let the solution $u$ be spherically symmetric, and $\displaystyle I = \bigcup_{j=j_1}^{j_2}I_j$ be a union of adjacent good intervals. Then there exists $j_1 \leq j_* \leq j_2$ such that $|I_{j_*}| \geq c\eta^{CC_0}|I|$.
\end{cor}

With the help of Corollary \ref{cor:3.6}, we can extract a finite sequence of distinct good subintervals that shrink exponentially in length and concentrate to a time $t_*$ with exponentially shrunken distance. See \cite[Proposition 3.8]{Tao05} for the proof.

\begin{prop} \label{cor:3.7-dec-seq}
Let the solution $u$ be spherically symmetric. Then there exist a time $t_* \in [T_-, T_+]$ and distinct good intervals $I_{j_1}, I_{j_2}, \cdots, I_{j_K}$ such that
\begin{itemize}
\item[(i)]  	$K$ satisfies $(c\eta^{C(C_0,C_1)})^K J \leq C\eta^{-CC_0}$,
\item[(ii)] 	$|I_{j_{k+1}}| \leq \dfrac{1}{2} |I_{j_k}|$ for every $1 \leq k \leq K-1$,
\item[(iii)]	$\mathrm{dist}(t_*, I_{j_k}) \leq C\eta^{-C(C_0, C_1)}|I_{j_k}|$ for every $1 \leq k \leq K$.
\end{itemize}
\end{prop}

We now finish the estimation of $J$.
Let $K$, $t_*$, and $I_{j_1}, I_{j_2}, \cdots, I_{j_K}$ be as in Proposition \ref{cor:3.7-dec-seq}. For each $1\leq k\leq K$, Proposition \ref{prop:3.3-locmass} admits a point $x_k \in \mathbb{R}^n$ such that
\[
\mathrm{Mass}\Big(u(t); B\big(x_{j_k}, \frac{C}{2}\eta^{-C'C_0}|I_{j_k}|^{1/2}\big)\Big) \geq c\eta^{C''C_0} |I_{j_k}|^{1/2}
\]
for every $t \in I_{j_k}$, where $C'$ is a much larger constant than $C''$. Applying \eqref{2.lm-ii} and \eqref{2.lm-i}, we obtain the lower and upper bounds of the local mass at time $t_*$ as
\begin{equation} \label{lm-est-f}
c\eta^{C(C_0, C_1)} |I_{j_k}|^{1/2} \leq
\mathrm{Mass}(u(t_*); B_k') \leq C\eta^{-C(C_0, C_1)} |I_{j_k}|^{1/2}
\end{equation}
where we write $B_k' = B\big(x_{j_k}, (1/2)C\eta^{-C(C_0, C_1)}|I_{j_k}|^{1/2}\big)$ for short.
Here, we repeat the same step with the domain of integration enlarged twice in radius, yielding
\begin{equation} \label{lm-est-f.x2}
c\eta^{C(C_0, C_1)} |I_{j_k}|^{1/2} \leq
\mathrm{Mass}(u(t_*); B_k) \leq C\eta^{-C(C_0, C_1)} |I_{j_k}|^{1/2}
\end{equation}
where $B_k = B(x_{j_k}, C\eta^{-C(C_0, C_1)}|I_{j_k}|^{1/2})$.
This additional step is not redundant since we have to mitigate potential problems that the smooth cutoff in the definition of local masses can cause in other inequalities ahead.

Let $N = C_4 \log(1/\eta)$. We observe from \eqref{lm-est-f} and \eqref{lm-est-f.x2} that
\begin{align*}
\sum_{m=k+N}^{K} \int_{B_{m}} |u(t_*)|^2 dx
&\leq	\sum_{m=k+N}^{K} \mathrm{Mass}(u(t_*); B_m)^2 \\
&\leq	C\eta^{-2C(C_0, C_1)} \sum_{m=k+N}^{K} |I_{j_m}| \\
&\leq	C'\eta^{-2C(C_0, C_1)} 2^{-N} |I_{j_{k}}| \\
&\leq	C''\eta^{-2C'(C_0, C_1)} 2^{-N} \mathrm{Mass}(u(t_*); B_k')^2 \\
&\leq	C''\eta^{-2C'(C_0, C_1)} 2^{-N} \int_{B_{k}} |u(t_*)|^2 dx
\end{align*}
where $C<C'<C''$ and $C(C_0, C_1)<C'(C_0, C_1)$ are other large absolute constants.
Setting $C_4$ large enough yet still as an absolute constant, we have
\[
\sum_{m=k+N}^{K} \int_{B_{m}} |u(t_*)|^2 dx
\leq	\frac{1}{2} \int_{B_{k}} |u(t_*)|^2 dx,
\]
and hence after subtracting from $\int_{B_{k}} |u(t_*)|^2 dx$ and applying \eqref{lm-est-f}, we have
\[
\int_{B_k \setminus \bigcup\limits_{m=k+N}^{K}B_m} |u(t_*)|^2 dx
\geq \frac{1}{2} \int_{B_k} |u(t_*)|^2 dx
\geq c\eta^{C(C_0, C_1)} |I_{j_k}|.
\]
Taking H\"{o}lder's inequality,
we get an $L_x^\frac{2n}{n-2}$ estimate below, independent of $|I_{j_k}|$.
\[
c\eta^{C(C_0, C_1)}
\leq \int_{B_k \setminus \bigcup\limits_{m=k+N}^{K}B_m} |u(t_*)|^\frac{2n}{n-2} dx
\]
Summation in $k$ then yields
\[
c\eta^{C(C_0, C_1)}K
\leq	\sum_{k=1}^{K} \int_{\bigcup\limits_{m=k}^{k+N-1}E_m} |u(t_*)|^\frac{2n}{n-2}dx
\leq	N\int_{\mathbb{R}^n} |u(t_*)|^\frac{2n}{n-2}dx
\leq	CE^CN
\]
where $E_k = B_k \setminus \bigcup\limits_{m=k+1}^{K}B_m$ ($k=1,2,\cdots,K$) are disjoint domains of integration. The multiplier $N$ explains the number of overlaps caused by the simple addition of the integrals within the supports $E_k$.

In conclusion, we obtain the bound of $K$,
\[
K
\leq CE^C \eta^{-C(C_0,C_1)}N
\leq C(C_0,C_1,C_4) \eta^{-C(C_0,C_1)}
\]
and then we extract the bound of $J$ by (i) of Proposition \ref{cor:3.7-dec-seq},
\[
J
\leq	C (c\eta^{C(C_0,C_1)})^{-K} \eta^{-CC_0}
\leq	C \exp\big(C'(C_0,C_1,C_4) \eta^{-C''(C_0,C_1)})
\]
and finally by substituting $\eta = 1/(C_3 E^{C_2})$ back, we finish the proof of Theorem \ref{thrm:main}.

\subsection{Proof of Proposition \ref{prop:3.3-locmass}} \label{Subsection:3.2}
We revisit and aim to fill in the missing proof of Proposition \ref{prop:3.3-locmass}.

We start with normalizing the subinterval $I_j$ into $I_j = [0,1]$ by the invariance to the time translation and scaling. We then divide it into two, $[0,1/2]$ and $[1/2,1]$, and by using the pigeonhole principle and the time reflection symmetry, we can also assume that
\[
\int_{1/2}^{1} \int_{\mathbb{R}^n} |u|^{\frac{2(n+2)}{n-2}} dx\,dt
\geq \frac{\eta}{2}.
\]
As $I_j = [0,1]$ is a good subinterval, we also have
\[ \label{0-1-good}
\int_{0}^{1} \int_{\mathbb{R}^n} |u_-|^{\frac{2(n+2)}{n-2}} dx\,dt
\leq \eta^{C_1}.
\]
Here, we claim that there exists $\eta^{C_0}\leq t_* \leq 1/2$ such that
\begin{gather}
\int_{\mathbb{R}^n} |u_-(t_*-\eta^{C_0})|^{\frac{2(n+2)}{n-2}}dx \leq C\eta^{C_1}, \label{pigeon-1} \\
\int_{t_*-\eta^{C_0}}^{t_*} \int_{\mathbb{R}^n} |u|^{\frac{2(n+2)}{n-2}}dx\,dt \leq C\eta^{C_0}. \label{pigeon-2}
\end{gather}

This claim is shown by contradiction in a complicated but similar way to the pigeonhole principle. Consider the set
\[
A = \Big\{ \eta^{C_0} \leq a \leq \frac{1}{2}\,:\,
\int_{a-\eta^{C_0}}^{a} \int_{\mathbb{R}^n} |u|^{\frac{2(n+2)}{n-2}}dx\,dt
> C\eta^{C_0} \Big\}.
\]
We cannot have $|A| > c > 0$ since we violate the partitioning condition \eqref{partitioned} after the summation on properly selected times in $A$. This implies $\big|[\eta^{C_0},1/2] \setminus A\big| > c > 0$ and shows \eqref{pigeon-2}. If we choose to deny \eqref{pigeon-1}, then we have
\[
\int_{\mathbb{R}^n} |u_-(a)|^{\frac{2(n+2)}{n-2}}dx > C'\eta^{C_1}
\]
for every $a \in [\eta^{C_0},1/2] \setminus A$, and $I_j=[0,1]$ becomes an exceptional subinterval for $C'>0$ taken large enough. This finishes the proof of the claim.

Meanwhile, by applying the Lemma \ref{lem:3.2} to the time interval $[t_*, 1]$, we see that
\begin{equation}
\label{pre-v-est-1}
\int_{t_*}^{1} \int_{\mathbb{R}^n} |e^{i(t-t_*)\Delta} u(t_*)|^{\frac{2(n+2)}{n-2}} dx\,dt
\geq c\eta^C,
\end{equation}
and by the definition of a good subinterval, we also see that
\begin{equation}
\label{pre-v-est-2}
\int_{t_*}^{1} \int_{\mathbb{R}^n} |u_-|^{\frac{2(n+2)}{n-2}} dx\,dt
\leq \eta^{C_1}.
\end{equation}
In addition, Strichartz estimates and \eqref{pigeon-2} yields
\begin{equation}
\label{pre-v-est-3}
\norm{L_t^{\frac{2(n+2)}{n-2}}([t_*,1],L_x^{\frac{2(n+2)}{n-2}})}
{\int_{t_*-\eta^{C_0}}^{t_*}e^{i(t-s)\Delta}(|x|^{-b}|u|^ku)(s) \,ds}
\leq CE^C \eta^{cC_0}.
\end{equation}

Here we let $\displaystyle v(t,x) = \int_{T_-}^{t_*-\eta^{C_0}}e^{i(t-s)\Delta}(|x|^{-b}|u|^ku)(s) \,ds$.
Applying \eqref{pre-v-est-1}, \eqref{pre-v-est-2}, and \eqref{pre-v-est-3} to the Duhamel's formula
\[
e^{i(t-t_*)\Delta} u(t_*)
=	u_-(t)
- i\int_{t_*-\eta^{C_0}}^{t_*}e^{i(t-s)\Delta}(|x|^{-b}|u|^ku)(s) \,ds
- iv(t)
\]
gives a lower bound of $v$ in the sense that
\begin{equation}
\label{v-lo-bd}
\norm{L_t^{\frac{2(n+2)}{n-2}}([t_*,1],L_x^{\frac{2(n+2)}{n-2}})} {v}
\geq	c\eta^C
\end{equation}
with the constants $C_0$, $C_1$, and $C_2$ taken large enough. Similarly, we find the following upper bound of $v$ in the Strichartz norm.
\begin{equation}
\label{v-up-bd}
\norm{\dot{S}^1([t_*, 1])} {v}
\leq	CE^C
\end{equation}

Define $u^{(h)}(t,x) = u(t, x+h)$, the spatial translate of $u$ by $-h\in\mathbb{R}^n$, and similarly $v^{(h)}(t,x) = v(t, x+h)$. We first estimate the $h$-difference form $v^{(h)} - v$ as in the following lemma, which is a variant of \cite[Lemma 3.4]{Tao05} and reminiscent of H\"{o}lder continuity.
The estimate itself can keep the same Lebesgue space form, but its proof, later provided in Subsection \ref{Subsection:3.3}, requires Lorentz spaces as intermediates.

\begin{lem} \label{lem:3.8-v-holder-est}
Let $n\geq 3$ and $0<b<\min(2,n/2)$, except for $n=4$ with $1<b<2$.
Then for every $h\in\mathbb{R}^n$, we have
\[
\norm {L_t^{\infty}([t_*,1],L_x^{\frac{2(n+2)}{n-2}})} {v^{(h)} - v}
\leq	CE^C \eta^{-CC_0} |h|^c.
\]
\end{lem}

Averaging Lemma \ref{lem:3.8-v-holder-est} over $|h| \leq r$, where $0 < r < 1$ is to be chosen shortly, we have
\[
\norm {L_t^{\infty}([t_*,1],L_x^{\frac{2(n+2)}{n-2}})} {v^{\rm av} - v}
\leq	CE^C \eta^{-CC_0} r^c
\]
where we define $\displaystyle v^{\rm av}(t,x) = \int_{\mathbb{R}^n} \chi(y) v(t,x+ry) \,dy$ and $\chi$ is a bump function as mentioned in Subsection \ref{Subsection:2.3}. Taking H\"{o}lder's inequality in time gives
\[
\norm {L_t^{\frac{2(n+2)}{n-2}}([t_*,1],L_x^{\frac{2(n+2)}{n-2}})} {v^{\rm av} - v}
\leq	CE^C \eta^{-CC_0} r^c,
\]
and by taking $C_2$ large enough, taking $r = \eta^{C'C_0}$ for $C'$ large enough, and considering \eqref{v-lo-bd} together, we obtain the following lower bound of $v^{\rm av}$.
\begin{equation}
\label{v-avg-lo-bd}
\norm{L_t^{\frac{2(n+2)}{n-2}}([t_*,1],L_x^{\frac{2(n+2)}{n-2}})} {v^{\rm av}}
\geq	c\eta^C
\end{equation}
Meanwhile, an upper bound of $v^{\rm av}$ is obtainable as a simple application of H\"{o}lder's inequality in time, Young's inequality in space, and coercivity.
\begin{equation}
\label{v-avg-up-bd}
\norm{L_t^{\frac{2n}{n-2}}([t_*,1],L_x^{\frac{2n}{n-2}})} {v^{\rm av}}
\leq	CE^C.
\end{equation}
Taking H\"{o}lder's inequality on \eqref{v-avg-lo-bd} and \eqref{v-avg-up-bd} in reverse order, we get
\[
\norm{L_t^{\infty}([t_*,1],L_x^{\infty})} {v^{\rm av}}
\geq c\eta^CE^{-C},
\]
and therefore, there exists a spacetime point $(t_j, x_j) \in [t_*, 1] \times \mathbb{R}^n$ such that
\[
\Big| \int_{\mathbb{R}^n} \chi(y) v(t_j, x_j+ry) \,dy \Big| \geq c\eta^CE^{-C},
\]
or in particular by H\"{o}lder's inequality again
\[
\mathrm{Mass}(v(t_j), B(x_j, R)) \geq c\eta^CE^{-C}r^C
\]
for every $R \geq r$. Noting that $v$ is a linear evolution on $[0,1]$ with energy $O(E^C)$ in view of \eqref{v-up-bd}, we can use \eqref{2.lm-ii} with $v$ in place of $u$ and deduce
\[
\mathrm{Mass}(v(t_*-\eta^{C_0}), B(x_j, R)) \geq c\eta^CE^{-C}r^C
\]
where we take $R = C\eta^{-C}E^{C}r^{-C}$ with large but absolute constants $C$.
(Since $r = \eta^{C'C_0}$, we now have $R = CE^C\eta^{-C-CC_0}$ with large but still absolute constants $C$.)

Therefore, considering the following bound of the local mass of $u_-$, which is obtained by applying H\"{o}lder's inequality at \eqref{pigeon-1},
\[
\mathrm{Mass}(u_-(t_*-\eta^{C_0}), B(x_j, R)) \leq CR^c\eta^{cC_1}
\]
we can claim as below with the constants $c$, $C$ adjusted if necessary,
\begin{equation} \label{u.lobd-lm}
\mathrm{Mass}(u(t_*-\eta^{C_0}), B(x_j, R)) \geq c\eta^CE^{-C}r^C = cE^{-C}\eta^{C+CC_0},
\end{equation}
in view of Duhamel's formula
\[
u(t_* - \eta^{C_0}) = u_-(t_* - \eta^{C_0}) - iv(t_* - \eta^{C_0})
\]
and by taking $C_1$ large enough depending on $C_0$.

Applying \eqref{2.lm-ii} to \eqref{u.lobd-lm}, we obtain the desired conclusion
\[
\mathrm{Mass}(u(t), B(x_j, R)) \geq cE^{-C}\eta^{C+CC_0}
\]
for every $t \in [0,1]$ and finish the proof of Proposition \ref{prop:3.3-locmass}.

\subsection{Proof of Lemma \ref{lem:3.8-v-holder-est}} \label{Subsection:3.3}
Lemma \ref{lem:3.8-v-holder-est} is one of the places where Lorentz spaces play crucial roles in this paper. We divide the cases by dimension $n$.

Consider $n\geq 5$. We see that the nonlinearity estimate below holds uniformly in time,
\begin{equation} \label{eq:nonl-leibniz}
\begin{aligned}
&\quad	\norm{L_x^{\frac{2n}{n+4},2}} {\nabla (|x|^{-b}|u|^ku)} \\
&\leq	C \norm{L_x^{\frac{n}{b+1}, \infty}} {|x|^{-(b+1)}}
\norm{L_x^{\frac{2n}{n-2}}} {u}^{\frac{2(2-b)}{n-2}}
\norm{L_x^{\frac{2n}{n-2},2}} {u}
+	C \norm{L_x^{\frac{n}{b}, \infty}} {|x|^{-b}}
\norm{L_x^{\frac{2n}{n-2}}} {u}^{\frac{2(2-b)}{n-2}}
\norm{L_x^{2}} {\nabla u} \\
&\leq	CE^C
\end{aligned}
\end{equation}
and hence by combining with the dispersive estimate of $e^{i(t-s)\Delta}$, we have
\begin{equation} \label{eq:nonl-disp-est,n-4}
\begin{aligned}
\norm{L_x^{\frac{2n}{n-4},2}} {\nabla e^{i(t-s)\Delta}(|x|^{-b}|u|^ku)(s)}
&\leq	C|t-s|^{-2}\norm{L_x^{\frac{2n}{n+4},2}} {\nabla (|x|^{-b}|u|^ku)(s)} \\
&\leq	CE^C|t-s|^{-2}.
\end{aligned}
\end{equation}
Since $L_x^{\frac{2n}{n-4},2}(\mathbb{R}^n) \hookrightarrow L_x^{\frac{2n}{n-4}}(\mathbb{R}^n)$,
by Minkowski's inequality for integrals to \eqref{eq:nonl-disp-est,n-4} in $s$,
we have
\begin{align*}
\sup_{t_* \leq t \leq 1} \norm{L_x^{\frac{2n}{n-4}}}{\nabla v}
&\leq	\sup_{t_* \leq t \leq 1} C \int_{T_-}^{t_*-\eta^{C_0}} \norm{L_x^{\frac{2n}{n-4}}} {\nabla e^{i(t-s)\Delta}(|x|^{-b}|u|^ku)(s)} ds \\
&\leq	CE^C \int_{-\infty}^{t_*-\eta^{C_0}} |t_*-s|^{-2} ds \\
&\leq	CE^C \eta^{-C_0}.
\end{align*}
Interpolating this with the coercivity
$ \norm{L_t^\infty([t_*,1];L_x^{\frac{2n}{n-2}})}{\nabla v} \leq CE^C $,
we have
\[
\norm{L_t^\infty([t_*,1],L_x^{\frac{2(n+2)}{n-2}})}{\nabla v} \leq CE^C \eta^{-cC_0}
\]
and the proof is completed by the Fundamental Theorem of Calculus.

We next consider $n=4$. The idea is similar to the case $n\geq 5$, but this time, we reach a bad endpoint $L^{\infty,2}(\mathbb{R}^n) = \{0\}$ in \eqref{eq:nonl-disp-est,n-4} and thus cannot use Lorentz spaces. Instead, the usual Lebesgue $L^{\infty}(\mathbb{R}^n)$ works well for $0<b\leq 1$. We also replace $L^{1,2}(\mathbb{R}^n)$ in \eqref{eq:nonl-leibniz} with $L^{1}(\mathbb{R}^n)$ accordingly.
\begin{align*}
\norm{L_x^{1}} {\nabla (|x|^{-b}|u|^ku)}
&\leq	C \norm{L_x^{\frac{4}{b+1}, \infty}} {|x|^{-(b+1)}}
\norm{L_x^{4,2}} {u}^{3-b}
+	C \norm{L_x^{\frac{4}{b}, \infty}} {|x|^{-b}}
\norm{L_x^{4,2}} {u}^{2-b}
\norm{L_x^{2}} {\nabla u} \\
&\leq	CE^C
\end{align*}
For $b > 1$, We note some problem in dividing the second Lorentz space exponents, although we disregard this case.

We lastly consider $n=3$.
Unlike $n\geq 4$, we have to directly handle the integrand of $v^{(h)}-v$ rather than differentiating it.
Once we find that
\begin{align*}
&\quad	\norm{L_x^{1}} {|x+h|^{-b}(|u|^ku)^{(h)} - |x|^{-b}|u|^ku} \\
&\leq	C \norm{L_x^{\frac{3}{b+1/2}}} {|x+h|^{-b} - |x|^{-b}}
\norm{L_x^{6}} {u}^{5-2b}
+	C \norm{L_x^{\frac{3}{b}, \infty}} {|x|^{-b}}
\norm{L_x^{6,2}} {u}^{4-2b}
\norm{L_x^{3,2}} {u^{(h)} - u} \\
&\leq	CE^C |h|^{1/2},
\end{align*}
the remaining process is similar to the case $n = 4$, except that we do not need the Fundamental Theorem of Calculus since we have already resolved all the necessary $h$-difference forms.

We only need to look at two more things. First, for every $0<\theta<1$, we have
\begin{equation} \label{holder-perturb-xpownb}
\norm{L_x^{\frac{n}{b+\theta},1}(\mathbb{R}^n)} {|x+h|^{-b} - |x|^{-b}} \leq C |h|^\theta
\end{equation}
where $C$ is an absolute constant depending only on $n$, $b$, and $\theta$. We can find an explicit bound for the decreasing rearrangement of $|x+h|^{-b} - |x|^{-b}$ to deduce \eqref{holder-perturb-xpownb}.
Second and last, we have
\[
\norm{L_x^{3,2}} {u^{(h)} - u}
\leq	C \norm{L_x^{2}} {u^{(h)} - u}^{1/2} \norm{L_x^{6,2}} {u^{(h)} - u}^{1/2}
\leq	C (CE^C|h|)^{1/2} (CE^C)^{1/2},
\]
which comes from the Fundamental Theorem of Calculus on the $L_x^{2}(\mathbb{R}^n)$ factor and Sobolev embedding on the $L_x^{6,2}(\mathbb{R}^n)$ factor.
This completes the proof of Lemma \ref{lem:3.8-v-holder-est}.

\section{Proof of Lemma \ref{lem:3.2} for Higher Dimensions} \label{Section:4}

This last section is dedicated to applying the arguments of Tao \cite[Appendix]{Tao05} to INLS. The usual bootstrapping arguments work directly on INLS and use Duhamel's formula singly, while these variant arguments work on the first spatial derivative of INLS and use Duhamel's formula iteratively by Neumann approximation. In this way, the previously unreachable higher-dimensional parameter setups, $n+2b \geq 6$, become reachable.

When $n=4$ or $5$, the effect of having both the low dimension and the low nonlinearity power leads to some technical difficulties; it is only possible to additionally resolve $5/4 < b < 2$ for $n=5$, and never for $n=4$. For $n\geq 6$, the dimension is high enough to eliminate the negative effect for every $0<b<2$.


We start with differentiating \eqref{inls} once in $x$ and obtain the following equation.
Compared to the standard NLS, we work with one more nonlinear term due to the non-constant multiplier $|x|^{-b}$.
\begin{equation} \label{inls-1st-deriv}
i\partial_t \nabla u + \Delta \nabla u
=	-b|x|^{-(b+2)}x \cdot |u|^k u
+ \frac{k+2}{2}|x|^{-b}|u|^{k} \nabla u
+ \frac{k}{2}|x|^{-b}|u|^{k-2}u^2 \overline{\nabla u}
\end{equation}
Fix $u$ and let $A$ be a linear operator defined on $\mathbb{C}^n$-valued functions $w \in \dot{S}^0([t_1,t_2])$ by
\[
Aw(t,x) := -i\int_{t_1}^{t} e^{i(t-s)\Delta} A_0w(s) \,ds
\]
where $A_0 w$ is defined as below. Note that $A_0 \nabla u$ is exactly the right side of \eqref{inls-1st-deriv}.
\begin{gather*}
A_0 w := A_{0}^{1} w + A_{0}^{2} w + A_{0}^{2'} w \\
\bigg(	A_{0}^{1} w := b|x|^{-(b+2)}x |u|^k \frac{1}{|\nabla|^2}\nabla\cdot w, \quad
A_{0}^{2} w := \frac{k+2}{2}|x|^{-b}|u|^{k} w, \quad
A_{0}^{2'} w := \frac{k}{2}|x|^{-b}|u|^{k-2}u^2 \overline{w}
\bigg)
\end{gather*}
Recalling the assumption $u_1(t,x) = e^{i(t-t_1)\Delta}u(t_1)$, we have a simplified Duhamel's formula
\[
\nabla u = \nabla u_1 + A\nabla u
\]
by the definition of $A$. We also observe that $A$ is a contraction mapping on $\dot{S}^0([t_1,t_2])$ in view of the estimate
\begin{equation} \label{A-contracts-S0}
\norm{\dot{S}^0([t_1,t_2])} {Aw} \leq C\eta^{\frac{2-b}{n+2}} \norm{\dot{S}^0([t_1,t_2])} {w}.
\end{equation}

The contraction property comes from the following estimate of $A_0 w$ followed by Strichartz estimates. To understand negative Sobolev-Lorentz spaces and the norm bounds in $\eta$, we recall the $L^{p,q}(\mathbb{R}^n)$ boundedness of Riesz transforms ($1<p<\infty$) and the assumption \eqref{lem3.2-assume} respectively. Reusing the exponents $p_\beta$, $q_\beta$, and $\tilde{q}_\beta$ defined in Subsection \ref{Subsection:2.2}, we see that
\begin{align*}
&\quad	\norm{L_t^{2} ([t_1,t_2]; L_x^{\frac{2n}{n+2},2})}
{A_0 w} \\
&\leq	C \norm{L_x^{\frac{n}{b+1}, \infty}} {|x|^{-(b+1)}}
\norm{L_t^{p_0}([t_1,t_2], L_x^{\tilde{q}_0})} {u}^{k}
\norm{L_t^{p_b}([t_1,t_2], \dot{W}_x^{-1;\tilde{q}_b,2})} {w} \\
&\quad	+ C \norm{L_x^{\frac{n}{b}, \infty}} {|x|^{-b}}
\norm{L_t^{p_0}([t_1,t_2], L_x^{\tilde{q}_0})} {u}^{k}
\norm{L_t^{p_b}([t_1,t_2], L_x^{q_b,2})} {w} \\
&\leq	C \eta^{\frac{2-b}{n+2}} \norm{\dot{S}^0([t_1,t_2])} {w}.
\end{align*}

Let $M$ be a positive integer. We write down the $M$\textsuperscript{th} order truncation of the Neumann series
\[
{\nabla u - \sum_{m=0}^{M-1} A^m\nabla u_1}
=	{A^M\nabla u_1}
+ {A\bigg(\nabla u - \sum_{m=0}^{M-1} A^m\nabla u_1\bigg)}.
\]
Since $\norm{\dot{S}^0([t_1,t_2])}{\nabla u_1} \leq CE^C$ by Strichartz estimate, we bootstrap on $\dot{S}^0([t_1,t_2])$ to have
\[
\norm{\dot{S}^0([t_1,t_2])} {\nabla u - \sum_{m=0}^{M-1} A^m\nabla u_1}
\leq CE^C (C\eta^{\frac{2-b}{n+2}})^M.
\]
Let $X = L_t^{\frac{2(n+2)}{n-2}}([t_1,t_2], \dot{W}_x^{-1,\frac{2(n+2)}{n-2}}(\mathbb{R}^n))$. Here, we take $M$ as the least integer greater than or equal to $4(n+2)/(2-b)$. By Sobolev embedding $\dot{S}^0([t_1,t_2]) \hookrightarrow X$ and taking $\eta$ possibly smaller, we can tidy the inequality as
\[
\norm{X} {\nabla u - \sum_{m=0}^{M-1} A^m\nabla u_1}
\leq C \eta.
\]
The assumption \eqref{lem3.2-assume} says $\norm{X}{\nabla u} \geq c\eta^{\frac{n-2}{2(n+2)}}$, so by subtraction we have
\begin{equation} \label{eq:lobd-sum-AmDu1}
c\eta^{\frac{n-2}{2(n+2)}}
\leq	\norm{X} {\sum_{m=0}^{M-1} A^m\nabla u_1}
\leq	\sum_{m=0}^{M-1} \norm{X} {A^m\nabla u_1}.
\end{equation}
At this point, we introduce one special lemma whose proof shall be given shortly.

\begin{lem} \label{lem:4.1}
There exists $0<\theta<1$ depending only on $n$ and $b$ such that
\[
\norm{X}{Aw}
\leq CE^C \norm{\dot{S}^{0}([t_1,t_2])}{w}^{1-\theta} \norm{X}{w}^{\theta}.
\]
for every $\mathbb{C}^n$-valued function $w \in \dot{S}^0([t_1,t_2])$.
\end{lem}

We iterate Lemma \ref{lem:4.1} with $w = A^m \nabla u_1$ over $m$. To simplify the estimating process, we note that $\norm{\dot{S}^0([t_1,t_2])}{A^m \nabla u_1} \leq CE^C$ uniformly in $m$. Eventually, we see that
\[
\norm{X} {A^m\nabla u_1}
\leq	(CE^C)^{(2-\theta)\sum_{\nu=0}^{m-1}\theta^{\nu}} \norm{X} {\nabla u_1}^{\theta^m}
\leq	(CE^C)^{\frac{2-\theta}{1-\theta}} \max\Big(\! \norm{X} {\nabla u_1}, \norm{X} {\nabla u_1}^{\theta^M} \!\Big)
\]
for every $m=0,1,\cdots,M-1$. Then \eqref{eq:lobd-sum-AmDu1} is turned into
\[
c\eta^{\frac{n-2}{2(n+2)}}
\leq	M(CE^C)^{\frac{2-\theta}{1-\theta}} \max\Big(\! \norm{X} {\nabla u_1}, \norm{X} {\nabla u_1}^{\theta^M} \!\Big)
\]
and therefore, we conclude with the following estimate for small $\eta$.
\[
\norm{X} {\nabla u_1}
\geq c' E^{-C'} \eta^{\frac{n-2}{2(n+2)} \theta^{-M}}
\geq c'' \eta^{\frac{n-2}{4(n+2)} \theta^{-M}}
\]
This finishes the proof of Lemma \ref{lem:3.2}, minus Lemma \ref{lem:4.1}.

\subsection{Proof of Lemma \ref{lem:4.1}} \label{Subsection:4.1}
We start with normalization $\norm{\dot{S}^0([t_1,t_2])}{w} = 1$ and $\norm{X}{w} = \alpha$ as in the proof of \cite[Lemma 4.1]{Tao05}. By Sobolev embedding, we have $\alpha \leq C$. It is enough to show that $\norm{X}{Aw} \leq CE^C \alpha^\theta$. Since
\begin{align*}
&\quad	\norm{L_t^{\frac{2(n+2)}{n-2}}([t_1,t_2], L_x^{\frac{2(n+2)}{n-2}}(\mathbb{R}^n))}{|\nabla|^{-1}Aw} \\
&\leq	\norm{L_t^{\infty}([t_1,t_2], L_x^{\frac{2n}{n-2}}(\mathbb{R}^n))}{|\nabla|^{-1}Aw}^{\frac{2}{n}} \norm{L_t^{\frac{2(n+2)}{n}}([t_1,t_2], L_x^{\frac{2n(n+2)}{n^2-2n-4}}(\mathbb{R}^n))}{|\nabla|^{-1}Aw}^{\frac{n-2}{n}} \\
&\leq	C \norm{L_t^{\infty}([t_1,t_2], L_x^{\frac{2n}{n-2}}(\mathbb{R}^n))}{|\nabla|^{-1}Aw}^{\frac{2}{n}} \norm{\dot{S}^0([t_1,t_2])}{Aw}^{\frac{n-2}{n}} \\
&\leq	C \norm{L_t^{\infty}([t_1,t_2], L_x^{\frac{2n}{n-2}}(\mathbb{R}^n))}{|\nabla|^{-1}Aw}^{\frac{2}{n}}
\end{align*}
by H\"{o}lder's inequality, Sobolev embedding and \eqref{A-contracts-S0}, we only have to show the existence of a small constant $0<\theta'<n/2$ such that for every $t_1 \leq t \leq t_2$,
\[
\norm{L_x^{\frac{2n}{n-2}}(\mathbb{R}^n)}{|\nabla|^{-1}Aw(t)} \leq CE^C \alpha^{\theta'}.
\]

We set $t=0$ by time translation and then let $v(x) = Aw(0, x)$. Let $\mathscr{M}$ be the Hardy-Littlewood maximal function
\[
\mathscr{M}f(x) = \sup_{r>0} \frac{1}{|B(x,r)|} \int_{B(x,r)}|f(y)| dy.
\]
It suffices to prove the following pointwise Hedberg-type inequality (See \citep{Hedberg72} for the inequalities involving maximal functions by Hedberg.)
\begin{equation} \label{D(-1)Aw-est-t=0,x}
\big||\nabla|^{-1}v\big|(x) \leq CE^C \alpha^{\theta'} (\mathscr{M}v(x))^{\frac{n-2}{n}}
\end{equation}
so that we can conclude by Hardy-Littlewood maximal inequality and \eqref{A-contracts-S0} that
\[
\norm{L_x^{\frac{2n}{n-2}}(\mathbb{R}^n)}{|\nabla|^{-1}v}
\leq	CE^C \alpha^{\theta'} \norm{L_x^2(\mathbb{R}^n)}{\mathscr{M}v}^{\frac{n-2}{n}}
\leq	CE^C \alpha^{\theta'}.
\]

We assume $0<\mathscr{M}v(x)<\infty$, since \eqref{D(-1)Aw-est-t=0,x} is trivial when $\mathscr{M}v(x)$ is either $0$ or $\infty$. From here, we also set $x=0$ by space translation. Let $R = (\mathscr{M}v(0))^{-2/n}$. We can rewrite \eqref{D(-1)Aw-est-t=0,x} as
\begin{equation} \label{D(-1)Aw-est-t=x=0}
\left| \int_{\mathbb{R}^n} \frac{v(y)}{|y|^{n-1}} dy \right|
\leq	CE^C \alpha^{\theta'} R^{-\frac{n-2}{2}}.
\end{equation}

We note that for every $r>0$,
\begin{equation} \label{Aw-loc-L1}
\int_{B(0,r)} |v(y)| \,dy \leq C \min(R^{-n/2}r^n, r^{n/2}).
\end{equation}
One upper bound $R^{-n/2}r^n$ comes from the definition of $\mathscr{M}$, and the other $r^{n/2}$ comes from H\"{o}lder's inequality and $\norm{L_x^2}{v} \leq \norm{\dot{S}^0([t_1,t_2])}{Aw} \leq 1$.
For a small absolute constant $0 < c_0 < 1$, with help of \eqref{Aw-loc-L1} and the layer cake decomposition, we can prove \eqref{D(-1)Aw-est-t=x=0} restricted on the region $|y| \leq \alpha^{c_0}R$ or $|y| \geq \alpha^{-c_0}R$. Therefore, after we dyadically decompose the remaining region $\alpha^{c_0}R \leq |y| \leq \alpha^{-c_0}R$ and take $c_0$ possibly smaller, it suffices to prove that
\[
\left| \int_{\mathbb{R}^n} v(y) \phi\Big(\frac{y}{r}\Big) dy \right|
\leq	CE^C \alpha^{\theta'} r^{n/2}
\]
for every $r>0$, where $\phi$ is a fixed real-valued bump function. We assume $r = 1$ as the general case can be obtained by the scaling $u(t,x) \mapsto r^{\frac{n-2}{2}}u(r^2 t, r x)$ and $w(t,x) \mapsto r^{\frac{n}{2}}w(r^2 t, r x)$.
Expanding $v(x)=Aw(0,x)$, our goal is rewritten as
\begin{equation} \label{int-Aw-finalstep}
\left| \int_{t_1}^{0} \int_{\mathbb{R}^n} \big(A_0^1w(t) +A_0^2w(t) +A_0^{2'}w(t)\big) e^{it\Delta} \phi \,dx \,dt \right|
\leq	CE^C \alpha^{\theta'}.
\end{equation}

We divide the left side of \eqref{int-Aw-finalstep} into two time regions, the distant past $[t_1, -\tau]$ and the recent past $[-\tau, 0]$, where $\tau = C\alpha^{-\theta''}$ with $\theta'' > 0$ determined later.
Estimating over the distant past is relatively easy, and even omissible when $t_1 \geq -\tau$.
After some computations, with the exponents $p_\beta$, $q_\beta$, and $\tilde{q}_\beta$ in Subsection \ref{Subsection:2.2} reused, we see that
\begin{align*}
&\quad	\left| \int_{t_1}^{-\tau} \int_{\mathbb{R}^n} A_0^1w(t) e^{it\Delta} \phi \,dx \,dt \right| \\
&\leq	C
\norm{L_x^{\frac{n}{b+1}, \infty}} {|x|^{-(b+1)}}
\norm{L_t^{p_0}([t_1, -\tau]; L_x^{\tilde{q}_0})} {u}^{k}
\norm{L_t^{p_b}([t_1, -\tau]; \dot{W}_x^{-1;\tilde{q}_b,2})} {w}
\norm{L_t^{2}([t_1, -\tau]; L_x^{\frac{2n}{n-2},2})} {e^{it\Delta} \phi} \\
&\leq	C\eta^{\frac{2-b}{n+2}} \norm{\dot{S}^0([t_1,-\tau])}{w} \tau^{-1/2} \norm{\dot{H}_x^1}{\phi} \\
&\leq	C \alpha^{\theta''/2}
\end{align*}
and similarly
\begin{align*}
&\quad	\left| \int_{t_1}^{-\tau} \int_{\mathbb{R}^n} A_0^2w(t) e^{it\Delta} \phi \,dx \,dt \right| \\
&\leq	C
\norm{L_x^{\frac{n}{b}, \infty}} {|x|^{-b}}
\norm{L_t^{p_0}([t_1, -\tau]; L_x^{\tilde{q}_0})} {u}^{k}
\norm{L_t^{p_b}([t_1, -\tau]; L_x^{q_b,2})} {w}
\norm{L_t^{2}([t_1, -\tau]; L_x^{\frac{2n}{n-2},2})} {e^{it\Delta} \phi} \\
&\leq	C\eta^{\frac{2-b}{n+2}} \norm{\dot{S}^0([t_1,-\tau])}{w} \tau^{-1/2} \norm{\dot{H}_x^1}{\phi} \\
&\leq	C \alpha^{\theta''/2}.
\end{align*}
The integral of $A_0^{2'}w(t) e^{it\Delta} \phi$ is estimated similarly to that of $A_0^2w(t) e^{it\Delta} \phi$.

It remains to prove \eqref{int-Aw-finalstep} over the recent past. We modify the $h$-difference approach used in the proof of Lemma \ref{lem:3.8-v-holder-est}.
Let $h \in \mathbb{R}^n$ and write $f^{(h)}(t,x) := f(t, x+h)$ for any measurable function $f(t,x)$.
At a fixed time $-\tau \leq t \leq 0$, we integrate the $h$-difference of $A_0^1w(t) e^{it\Delta}\phi$ but with $w$ left unperturbed.
For any absolute constant $0 < \delta < 1$, we have
\begin{equation} \label{A0.1-perturb}
\begin{aligned}
&\quad	\left| \int_{\mathbb{R}^n}
\Big(\frac{x+h}{|x+h|^{b+2}} |u^{(h)}|^k(t) e^{it\Delta} \phi^{(h)} - \frac{x}{|x|^{b+2}} |u|^k(t) e^{it\Delta} \phi \Big) \frac{\nabla\cdot w}{|\nabla|^2} \,dx \right| \\
&\leq	C
\norm{L_x^{\frac{n}{b+1+\delta}}} {\frac{x+h}{|x+h|^{b+2}} - \frac{x}{|x|^{b+2}}}
\norm{L_x^{\frac{2n}{n-2}}} {u}^{k}
\norm{\dot{W}_x^{-1,\frac{2n}{n-2}}} {w}
\norm{L_x^{\frac{2n}{n-4-2\delta}}} {e^{it\Delta} \phi} \\
&\quad	+C
\norm{L_x^{\frac{n}{b+1},\infty}} {|x|^{-(b+1)}}
\norm{L_x^{\frac{2}{k}}} {|u^{(h)}|^k - |u|^k}
\norm{\dot{W}_x^{-1,\frac{2n}{n-2}}} {w}
\norm{L_x^{\frac{2n(n-2)}{n(n-6)+4b},1}} {e^{it\Delta} \phi} \\
&\quad	+C
\norm{L_x^{\frac{n}{b+1},\infty}} {|x|^{-(b+1)}}
\norm{L_x^{\frac{2n}{n-2}}} {u}^{k}
\norm{\dot{W}_x^{-1,\frac{2n}{n-2}}} {w}
\norm{L_x^{\frac{2n}{n-4},1}} {e^{it\Delta} (\phi^{(h)} - \phi)} \\
&\leq	CE^C \langle t\rangle^{-2} \norm{L_t^{\infty}([-\tau, 0], \dot{W}_x^{-1,\frac{2n}{n-2}})}{w}
(|h|^{\delta} + |h|^{k} + |h|) \\
&\leq	CE^C \langle t\rangle^{-2} |h|^{\min(\delta, k)}
\end{aligned}
\end{equation}
provided that $n-4-2\delta \geq 0$ and $n(n-6)+4b > 0$.

We note a few things to justify \eqref{A0.1-perturb}. First, we can obtain
\[
\norm{L_x^{\frac{n}{b+1+\delta}}} {\frac{x+h}{|x+h|^{b+2}} - \frac{x}{|x|^{b+2}}}
\leq C(n,b,\delta)|h|^\delta < \infty
\]
as a variant of \eqref{holder-perturb-xpownb}. Second, we also observe that
\[
\norm{L_x^{\frac{2}{k}}} {|u^{(h)}|^k - |u|^k}
\leq C \norm{L_x^{2}} {u^{(h)} - u}^k
\leq CE^C |h|^k,
\]
which comes from the pointwise inequality $\big| |u^{(h)}|^k - |u|^k \big| \leq C |u^{(h)} - u|^k$ (since $0 < k \leq 1$ when $n+2b \geq 6$) and by the Fundamental Theorem of Calculus on $u^{(h)} - u$.
Lastly, all the factors involving $e^{it\Delta} \phi$ are estimable by the dispersive estimate and the smoothness of $\phi$.
All three Lorentz space exponents associated with $e^{it\Delta} \phi$,
\[
2n/(n-4), \ 2n/(n-4-2\delta), \text{ and } 2n(n-2)/(n(n-6)+4b),
\] should lie in $[2n/(n-4), \infty)$ to support the time decay $\langle t \rangle^{-2}$. This is possible when $0<\delta \leq 1/2$ and either of $n\geq 6$ with $0<b<2$ or $n=5$ with $5/4<b<2$ holds. ($2n/(n-4-2\delta) = \infty$ is permitted.) For a similar reason, our arguments in this section are discontinued when $n=4$.

We next estimate the integral of the $h$-difference of $A_0^2w(t) e^{it\Delta} \phi$.
Under the same conditions on $n$, $b$, and $\delta$, for every $-\tau \leq t \leq 0$, we have
\begin{equation} \label{A0.2-perturb}
\begin{aligned}
&\quad	\left| \int_{\mathbb{R}^n}
\Big(\frac{1}{|x+h|^{b}} |u^{(h)}|^k(t) e^{it\Delta} \phi^{(h)} - \frac{1}{|x|^{b}} |u|^k(t) e^{it\Delta} \phi \Big) \overline{w} \,dx \right| \\
&\leq	C
\norm{L_x^{\frac{n}{b+\delta}}} {\frac{1}{|x+h|^{b}} - \frac{1}{|x|^{b}}}
\norm{L_x^{\frac{2n}{n-2}}} {u}^{k}
\norm{L_x^2} {w}
\norm{L_x^{\frac{2n}{n-4-2\delta}}} {e^{it\Delta} \phi} \\
&\quad	+C
\norm{L_x^{\frac{n}{b},\infty}} {|x|^{-b}}
\norm{L_x^{\frac{2}{k}}} {|u^{(h)}|^k - |u|^k}
\norm{L_x^2} {w}
\norm{L_x^{\frac{2n(n-2)}{n(n-6)+4b},1}} {e^{it\Delta} \phi} \\
&\quad	+C
\norm{L_x^{\frac{n}{b},\infty}} {|x|^{-b}}
\norm{L_x^{\frac{2n}{n-2}}} {u}^{k}
\norm{L_x^2} {w}
\norm{L_x^{\frac{2n}{n-4},1}} {e^{it\Delta} (\phi^{(h)} - \phi)} \\
&\leq	CE^C \langle t\rangle^{-2} \norm{L_t^{\infty}([-\tau, 0], L_x^2)}{w}
(|h|^{\delta} + |h|^{k} + |h|) \\
&\leq	CE^C \langle t\rangle^{-2} |h|^{\min(\delta, k)}.
\end{aligned}
\end{equation}
The integral is similarly estimable when the integrand involves $A_0^{2'}w$ in place of $A_0^{2}w$.

Adding the estimates \eqref{A0.1-perturb}, \eqref{A0.2-perturb} and its $A_0^{2'}w$ version, we have
\[
\left| \int_{\mathbb{R}^n} A_0(w-w^{(-h)}) e^{it\Delta} \phi \,dx \right|
\leq	CE^C \langle t\rangle^{-2} |h|^{\min(\delta, k)}
\]
for every $-\tau \leq t \leq 0$. Taking the average on $|h|<r$ here, we have
\begin{equation} \label{A0-avg-perturb}
\left| \int_{\mathbb{R}^n} A_0(w-w^{\rm av}) e^{it\Delta} \phi \,dx \right|
\leq	CE^C r^{\min(\delta, k)} \langle t\rangle^{-2}
\end{equation}
where we define $\displaystyle w^{\rm av}(t,x) := \int_{\mathbb{R}^n} \chi(y) w(t, x+ry) \,dy$ and $0<r<1$ is determined later.

Meanwhile, Young's inequality applied to $\displaystyle w^{\rm av}(t,x) = \int_{\mathbb{R}^n} \frac{1}{r^n} \chi\Big(\frac{y'-x}{r}\Big) w(t,y') \,dy'$ yields
\[
\norm{L_x^{\frac{2(n+2)}{n-2}}} {w^{\rm av}}
\leq	Cr^{-1} \norm{\dot{W}_x^{-1,\frac{2(n+2)}{n-2}}} {w}
\]
and hence for every $-\tau \leq t \leq 0$,
\begin{equation} \label{A0-avg}
\begin{aligned}
&\quad	\left| \int_{\mathbb{R}^n} A_0w^{\rm av} e^{it\Delta} \phi \,dx \right| \\
&\leq	C
\norm{L_x^{\frac{n}{b},\infty}} {|x|^{-b}}
\norm{L_x^{\frac{2n}{n-2}}} {u}^{k}
\norm{L_x^{\frac{2(n+2)}{n-2}}} {w^{\rm av}}
\norm{L_x^{\frac{2n(n+2)}{(n+4)(n-2)},1}} {e^{it\Delta} \phi} \\
&\leq	CE^C r^{-1} \langle t\rangle^{-\frac{4}{n-2}} \norm{\dot{W}_x^{-1,\frac{2(n+2)}{n-2}}} {w}.
\end{aligned}
\end{equation}
Therefore, adding the two estimates \eqref{A0-avg-perturb} and \eqref{A0-avg} integrated in time, we have
\[
\left| \int_{-\tau}^{0} \int_{\mathbb{R}^n} A_0w(t) e^{it\Delta} \phi \,dx \,dt \right| \\
\leq	CE^C (r^{\min(\delta, k)} + r^{-1}\tau^{\frac{n+6}{2(n+2)}}\alpha)
\]
where we recall that $\tau = C\alpha^{-\theta''} \geq C$. Taking $\theta''>0$ small enough, we can optimize the last inequality in $r$ to finish the proof of \eqref{int-Aw-finalstep} and hence Lemma \ref{lem:4.1}.

\bibliographystyle{plain}
\bibliography{INLS-scatter-ref}

\end{document}